\documentclass[11pt]{amsart}
\usepackage{multirow,bigdelim}
\usepackage{amsfonts}
\usepackage{dsfont}
\usepackage{amsmath}
\usepackage{cases}
\usepackage{mathrsfs}
\usepackage{hyperref}
\usepackage{multimedia}
\usepackage{graphicx}
\usepackage{indentfirst}
\usepackage{bm}
\usepackage{amsthm}
\usepackage{mathrsfs}
\usepackage{latexsym}
\usepackage{amsmath}
\usepackage{amssymb}
\usepackage{microtype}
\usepackage{relsize}

\usepackage{hyperref}

\DeclareSymbolFont{SY}{U}{psy}{m}{n}
\DeclareMathSymbol{\emptyset}{\mathord}{SY}{'306}

\theoremstyle{plain}

\newtheorem{thm}{Theorem}[section]
\newtheorem{corollary}[thm]{Corollary}
\newtheorem{lemma}[thm]{Lemma}

\newtheorem{definition}[thm]{Definition}

\theoremstyle{definition}

\numberwithin{equation}{section}
\begin{document}
\title[$K_{0}$-groups and strongly irreducible decompositions of operator tuples]{$K_{0}$-groups and strongly irreducible decompositions of operator tuples}

\author{Jing Xu}
\email{xujingmath@outlook.com}
\address{School of Mathematical Sciences, Hebei Normal University, Shijiazhuang, Hebei, 050016, China}

\thanks{This work was supported by the National Natural Science Foundation of China, Grant No. 11922108 and 12001159.}



\begin{abstract}
An operator tuple $\mathbf{T}=(T_{1},\ldots,T_{n})$ is called strongly irreducible (SI), if the joint commutant of $\mathbf{T}$ does not any nontrivial idempotent operator.
In this paper, we study the uniqueness of finitely strong irreducible decomposition of operator tuples up to similarity by  $K$-theory of operator algebra,  and give the algebraically similarity invariants of the Cowen-Douglas tuple with index 1 by using $K_{0}$-group of the commutant of operator tuples. As an application, we calculate $K_{0}$-groups of some multiplier algebras, and describe the similarity of backwards multishifts on Drury-Arveson space by means of inflation theory.
\end{abstract}

\maketitle

\section{Introduction}
For a complex separable Hilbert space $\mathcal{H}$, let $\mathcal{L}(\mathcal{H})$ denote the algebra of bounded linear operators on $\mathcal{H}$.
Given a complex separable Hilbert space $\mathcal{H}$, let ${\mathcal L}({\mathcal H})$ denote the algebra of bounded linear operators on $\mathcal H$.
Recall that operators $T$ and $S$ are said to be unitary equivalent if there is a unitary operator $U$ such that $T=U^{*}SU$ (denoted by $T\sim_{u}S$). The operators $T$ is similar to $S$ if there is a bounded invertible operator $X$ satisfying $T=X^{-1}SX$ (denoted by $T\sim_{s}S$).

An operator $T\in\mathcal{L}(\mathcal{H})$ is called \emph{strongly irreducible} if $T$ does not commute with any nontrivial idempotent operator. If idempotent operator is replaced by self adjoint idempotent, then $T$ is said to be \emph{irreducible}, see \cite{Gilfeather,Halmos1968,Jiang1979}. The concepts of strong irreducibility and of Banach irreducibility introduced by Gilfeather and Jiang in \cite{Gilfeather} and \cite{JS2006}, respectively, turned out to be equivalent. Strong irreducibility is preserved by similarity, and irreducibility is preserved by unitary equivalent.

The Jordan canonical form theorem states that every operator on a finite-dimensional Hilbert space $\mathcal{H}$ can be uniquely written as a direct sum of strongly irreducible operators up to similarity. Is there a corresponding analogue when one considers operators on an infinite-dimensional complex separable Hilbert space $\mathcal{H}$?
The notion of a unicellular operator was introduced by Brodski\v{1} in \cite{MSB1956,MSB1968} and Kisilevs'ki\v{1} was shown in \cite{GEK19672, GEK1967} that dissipative operators can be written as a direct sum of unicellular operators. In \cite{Halmos1968},  Halmos proved that the set of irreducible operators was dense in $\mathcal{L}(\mathcal{H})$ in the sense of Hilbert-Schmidt norm approximations. In \cite{BNF1979, BNF1975, NF19702, NF19703, NF1972, NF}, Bercovici, Foia\c{s}, K\'{e}rchy and Sz.-Nagy showed that every operator of class $C_{0}$ on a complex separable Hilbert space was proven to be similar to a Jordan operator. In \cite{DH1990}, Herrero show that every bitriangular operator was to be quasisimilar to a Jordan operator.

For $m \in \mathbb{N}_{+}$, let $\mathbf{T}=(T_{1},\ldots,T_{m}) \in \mathcal{L}(\mathcal{H})^m$ and $\mathbf{S}=(S_{1},\ldots,S_{m})\in \mathcal{L}(\mathcal{H})^m$ be $m$-tuples of commuting operators on $\mathcal{H}$. If there is a unitary operator $U\in \mathcal{L}(\mathcal{H})$ such that $U\mathbf{T}=\mathbf{S}U$, i.e., $UT_{i}=S_{i}U$ for all $i, 1\leq i\leq m$, then $\mathbf{T}$ and $\mathbf{S}$ are said to be \emph{unitarily equivalent} (denoted by $\mathbf{T}\sim_{u}\mathbf{S}$). If there is an invertible operator $X\in \mathcal{L}(\mathcal{H})$ such that $X\mathbf{T}=\mathbf{S}X$, i.e., $XT_{i}=S_{i}X$ for all $i, 1\leq i\leq m$, then $\mathbf{T}$ and $\mathbf{S}$ are \emph{similar} (denoted by $\mathbf{T}\sim_{s}\mathbf{S}$).

To study equivalence problems for bounded linear operators on infinite-dimensional complex separable Hilbert space to which standard methods do not apply, Cowen and Douglas introduced in the late 1970's, a class of operators with a holomorphic eigenvector bundle structure \cite{CD,CD2}. For $\mathbf{T}=(T_1, \ldots, T_m) \in \mathcal{L}(\mathcal{H})^m$ be $m$-tuples of commuting operators on $\mathcal{H}$. The tuple of commuting opartor $\mathbf{T}: \mathcal{H}\longrightarrow\mathcal{H}\oplus\cdots\oplus\mathcal{H}$ is defined by
$$\mathbf{T}(x)=(T_{1}x,\ldots,T_{m}x),\quad x\in \mathcal{H}.$$
For $w=(w_{1},\ldots,w_{m}) \in \Omega$, let $\mathbf{T}-w=(T_{1}-w_{1},\ldots,T_{m}-w_{m})$, $\ker(\mathbf{T}-w)=\bigcap \limits_{i=1}^{m}\ker(T_{i}-w_{i}),$
and $\mathcal{A}^{\prime}(T_{i})=\{X\in\mathcal{L}(\mathcal{H}):XT_{i}=T_{i}X\}$ and $\mathcal{A}^{\prime}(\mathbf{T})=\bigcap\limits_{i=1}^{m}\mathcal{A}^{\prime}(T_{i})$ be commutants of $T_{i}$ and $\mathbf{T}$, respectively.

\begin{definition}\cite{CD,CD2}
For $\Omega$ a connected open subset of $\mathbb{C}^{m}$ and $n$ a positive integer, let $\mathbf{\mathcal{B}}_{n}^{m}(\Omega)$ denotes the Cowen-Douglas class of $m$-tuples $\mathbf{T}=(T_{1},\ldots,T_{m})\in\mathcal{L}(\mathcal{H})^{m}$ satisfying:
\begin{itemize}
  \item [(1)]$\text{ran\,}(\mathbf{T}-w)$ is closed for all $w$ in $\Omega$;
  \item [(2)]$\bigvee \limits_{w{\in}{\Omega}} \ker(\mathbf{T}-w)=\mathcal H$; and
  \item [(3)]$\dim \ker (\mathbf{T}-w)=n$ for $w$ in $\Omega$,
\end{itemize}
where $\bigvee$ denotes the closed linear span.
\end{definition}
For an $m$-tuple of commuting operators $\mathbf{T}=(T_{1},\ldots,T_{m}) \in \mathbf{\mathcal{B}}_{n}^{m}(\Omega),$ Cowen and Douglas proved in \cite{CD, CD2} that an associated holomorphic eigenvector bundle $\mathcal{E}_{\mathbf{T}}$ over $\Omega$ of rank $n$ exists, where
$$\mathcal{E}_\mathbf{T}=\{(w, x)\in \Omega\times \mathcal
H: x \in \ker(\mathbf{T}-w)\},\quad \pi(w,x)=w.$$
Furthermore, it was shown that two operator tuples $\mathbf{T}$ and $\widetilde{\mathbf{T}}$ in $\mathbf{\mathcal{B}}_{n}^{m}(\Omega)$ are unitarily equivalent if and only if the vector bundles $\mathcal{E}_\mathbf{T}$ and $\mathcal{E}_{\widetilde{\mathbf{T}}}$ are equivalent as Hermitian holomorphic vector bundles. They also showed that every $m$-tuple
$\mathbf{T}\in\mathcal{B}_{n}^{m}(\Omega)$ can be realized as the adjoint of an $m$-tuple of multiplication operators by the coordinate functions on a Hilbert space of holomorphic functions on $\Omega^{*}=\{\overline{w}: w \in \Omega\}$, which can also be seen Zhu's \cite{Zhu2000} and Eschmeier and Schmitt's \cite{CS}, respectively.
In \cite{CS},
Curto and Salinas established a relationship between the class $\mathcal{B}_{n}^{m}(\Omega)$
and generalized reproducing kernels to describe when two $m$-tuples are unitarily equivalent. A similarity result for Cowen-Douglas operators in geometric terms such as curvature had been much more difficult to obtain. In fact, the work of Clark and Misra in \cite{CM2, CM} showed that the Cowen-Douglas conjecture that similarity can be determined from the behavior of the quotient of the entries of curvature matrices was false.
For the class $\mathcal{B}_{n}^{1}(\Omega)$, the work Cao, Fang, Jiang \cite{CFJ2002}, Jiang \cite{J2004}, and Jiang, Guo, Ji \cite{JGJ} involve the $K_{0}$-group of the commutant algebra as an invariant to show that an operator in $\mathcal{B}_{n}^{1}(\Omega)$ has a unique strong irreducible decomposition up to similarity.

In this paper, we study the uniqueness of finitely strong irreducible decomposition of operator tuples up to similarity by  $K$-theory of operator algebra,  and give the algebraically similarity invariants of the Cowen-Douglas tuple with index 1 by using $K_{0}$-group of the commutant of operator tuples. As an application, we calculate $K_{0}$-groups of some multiplier algebras, and describe the similarity of backwards multishifts on Drury-Arveson space by means of inflation theory.

\section{preliminaries}
\begin{definition}
An $m$-tuple $\mathbf{T}=(T_{1},T_{2},\ldots,T_{m})\in\mathcal{L}(\mathcal{H})^{m}$ is called strongly irreducible, if $\mathcal{A}^{\prime}(\mathbf{T})=\bigcap\limits_{i=1}^{m}\mathcal{A}^{\prime}(T_{i})$ the commutant of $\mathbf{T}$ does
not any nontrivial idempotent operator. If there is no self-adjoint idempotent in $\mathcal{A}^{\prime}(\mathbf{T})$, we call $\mathbf{T}$ irreducible.
\end{definition}

\begin{definition}\label{def1}
Let $\mathbf{T}=(T_{1},T_{2},\ldots,T_{m})\in\mathcal{L}(\mathcal{H})^{m}$ be an $m$-tuple and $\mathcal{P}=\{P_{i}\}_{i=1}^{n}, n<\infty,$ be a set of idempotents.
$\mathcal{P}=\{P_{i}\}_{i=1}^{n}$ is called a unit finite decomposition of $\mathbf{T}$ if it satisfies:
\begin{itemize}
  \item [(1)]$P_{i}\in\mathcal{A}^{\prime}(\mathbf{T})=\bigcap\limits_{i=1}^{m}\mathcal{A}^{\prime}(T_{i})$;
  \item [(2)]$P_{i}P_{j}=0$ for $i\neq j$;
  \item [(3)]$\sum\limits_{i=1}^{n}P_{i}=I$, where $I$ is the identity operator.
\end{itemize}
In addition, if
\begin{itemize}
  \item [(4)]$\mathbf{T}|_{P_{i}\mathcal{H}}=(T_{1}|_{P_{i}\mathcal{H}},\ldots,T_{m}|_{P_{i}\mathcal{H}})$ is strongly irreducible for $i=1,2,\ldots,n$,
\end{itemize}
we call $\mathcal{P}=\{P_{i}\}_{i=1}^{n}$ is a unit finite strongly irreducible decomposition of $\mathbf{T}$.
\end{definition}

\begin{definition}
Let $\mathbf{T}=(T_{1},\ldots,T_{m})\in\mathcal{L}(\mathcal{H})^{m}$ be an $m$-tuple. We say that $\mathbf{T}$ has finite strongly irreducible decomposition, if for any idempotent $P$ in $\mathcal{A}^{\prime}(\mathbf{T})$, $\mathbf{T}|_{P\mathcal{H}}$ has a unit finite strongly irreducible decomposition.
\end{definition}

\begin{definition}
Let $m$-tuple $\mathbf{T}=(T_{1},\ldots,T_{m})\in\mathcal{L}(\mathcal{H})^{m}$ has finite strongly irreducible decomposition. If $\mathcal{P}=\{P_{i}\}_{i=1}^{n}$ and $\mathcal{Q}=\{Q_{i}\}_{i=1}^{k}$ are any two unit finite strongly irreducible decomposition of $\mathbf{T}$ and satisfy the following:
\begin{itemize}
  \item [(1)]$n=k$;
  \item [(2)]there exists an operator $X\in GL(\mathcal{A}^{\prime}(\mathbf{T}))=\{Y: Y$ is invertible in $\mathcal{A}^{\prime}(\mathbf{T})\}$ and a permutation $\Pi\in \mathcal{S}_{n}$ such that $XP_{i}X^{-1}=Q_{\Pi(i)}$ for $i=1,2,\ldots,n$.
\end{itemize}
Then we say that $\mathbf{T}$ has unique finite strongly irreducible decomposition up to similarity.
\end{definition}

\begin{definition}\cite{Agler1988}
Let $m\geq 1$. A family is a collection $\mathcal{F}$ of $m$-tuples $\mathbf{T}=(T_{1},\ldots,T_{m})$ of Hilbert space operators, $T_{i}\in\mathcal{L}(\mathcal{H}),$ such that:
\begin{itemize}
  \item [(1)]$\mathcal{F}$ is bounded, i.e. there exists $c>0$ such that for all $\mathbf{T}=(T_{1},\ldots,T_{m})\in\mathcal{F}$ we have $\Vert T_{i}\Vert\leq c$ for all $i=1,\ldots,m,$
  \item [(2)]$\mathcal{F}$ is preserved under restrictions on invariant subspace, i.e. whenever $\mathbf{T}=(T_{1},\ldots,T_{m})\in\mathcal{F}$ and $\mathcal{M}\subseteq \mathcal{H}$ such that $T_{i}\mathcal{M}\subseteq\mathcal{M}$ for all $i$, then $T|\mathcal{M}\in\mathcal{F},$
  \item [(3)]$\mathcal{F}$ is preserved under direct sums, i.e. whenever $\mathbf{T}_{n}\in\mathcal{F}$ is a sequence of $m$-tuples, then $\oplus_{n}\mathbf{T}_{n}\in\mathcal{F},$
  \item [(4)]$\mathcal{F}$ is preserved under unital $\ast$-representations, i.e. if $\pi:\mathcal{L}(\mathcal{H})\longrightarrow\mathcal{L}(\mathcal{K})$ is a $\ast$-homomorphism with $\pi(I)=I$ and if $\mathbf{T}=(T_{1},\ldots,T_{m})\in\mathcal{F}$, then $\pi(\mathbf{T})=(\pi(T_{1}),\ldots,\pi(T_{m}))\in\mathcal{F}$.
\end{itemize}
\end{definition}

\begin{definition}\cite{Agler1988}
An $m$-tuple $\mathbf{T}=(T_{1},\ldots,T_{m})$ is called a spherical isometry if $\sum\limits_{i=1}^{m}\Vert T_{i}x\Vert^{2}=\Vert x\Vert^{2}$ for every $x\in\mathcal{H},$ that is to say,
$\sum\limits_{i=1}^{m}T_{i}^{*}T_{i}=I$.
\end{definition}

\begin{definition}\cite{Agler1988}
An $m$-tuple $\mathbf{U}=(U_{1},\ldots,U_{m})$ is called a spherical unitary if $\sum\limits_{i=1}^{m}U_{i}^{*}U_{i}=I$ and each $U_{i}$ is a normal operator.
\end{definition}

\begin{definition}
For $n\in\mathbb{N}$, an $m$-tuple $T=(T_{1},\ldots,T_{m})\in\mathcal{L}(\mathcal{H})^{m}$ is said to be an n-hypercontraction. if
$$(I_{\mathcal{H}}-T^{*}_{1}T_{1}-\cdots-T^{*}_{m}T_{m})^{k}\geq0$$
for all integers $k, 1\leq k\leq n$.
\end{definition}

\subsection{$K_{0}$-group of a unital Banach algebra}
Let $\mathcal{A}$ be a unital Banach algebra, and $a,b$ be idempotents in $\mathcal{A}.$ We denoted that $a$ and $b$ algebraic equivalence $(a\sim b)$ if there exist $x, y\in \mathcal{A}$ such that $xy=a$ and $yx=b$. We write $a\sim_{s}(\mathcal{A}) b$ if there exists a $z\in GL(\mathcal{A})$ with $zaz^{-1}=b.$ Let $M_{n}(\mathcal{A})$ be the set of all $n\times n$ matrices
\begin{equation}\begin{array}{lll}
\begin{pmatrix}\begin{smallmatrix}a_{1,1}&a_{1,2}&\cdots&a_{1,n}\\ a_{2,1}&a_{2,2}&\cdots&a_{2,n}\\ \vdots&\vdots&\ddots&\vdots\\a_{n,1}&a_{n,2}&\cdots&a_{n,n}\end{smallmatrix}\end{pmatrix},\notag
\end{array}\end{equation}
where each matric entry $a_{i,j}$ in $ \mathcal{A}$.
Set $$M_{\infty}(\mathcal{A})=\bigcup\limits_{n=1}^{\infty}M_{n}(\mathcal{A}).$$

\begin{definition}\cite{MFN}
$\mathcal{P}(\mathcal{A})$ is the set of algebraic equivalence classes of idempotents in $\mathcal{A}$ and $\bigvee(\mathcal{A})=\mathcal{P}(M_{\infty}(\mathcal{A}))$.
\end{definition}

From the classical results of K-theory, one obtains exactly the same semigroup starting with $\sim_{s}$ instead of $\sim$, since the two notions coincide on $M_{\infty}(\mathcal{A}).$

\begin{definition}\cite{MFN}
$K_{0}(\mathcal{A})$ is the Grothendieck group of $\bigvee(\mathcal{A})$.
\end{definition}

\subsection{The Cowen-Douglas Class}

\begin{lemma}
Let $\mathbf{T}=(T_{1},T_{2},\ldots,T_{m})\in\mathcal{B}_{n}^{m}(\Omega)$ and $\mathbf{T}$ be unitary equivalent to the adjoint
of $m$-tuple $\mathbf{M}_{z}=(M_{z_{1}},\ldots,M_{z_{m}})$ of multiplication operators on analytic function space $\mathcal{H}_{K}$ with reproducing kernel
$$K(z,w)=\sum\limits_{\alpha\in\mathbb{N}^{m}}\widehat{f}(\alpha)z^{\alpha}\overline{w}^{\alpha},$$
where $z,w\in\Omega$ and $\widehat{f}(\alpha)>0$ for all $\alpha\in\mathbb{N}^{m}$.
Then $\mathbf{T}$ is unitary equivalent to $m$-tuple of commuting weighted backward shifts with weight sequence $\left\{\sqrt{\frac{\widehat{f}(\alpha)}{\widehat{f}(\alpha+e_{1})}},\ldots,\sqrt{\frac{\widehat{f}(\alpha)}{\widehat{f}(\alpha+e_{m})}}\right\}_{\alpha\in\mathbb{N}^{m}}$.
\end{lemma}

\begin{proof}
Let $\{\mathbf{e}_{\alpha}\}_{\alpha\in\mathbb{N}^{m}}$ be the orthonormal basis of space $\mathcal{H}_{K}$. We have that
$$K(z,w)=\sum\limits_{\alpha\in\mathbb{N}^{m}}\widehat{f}(\alpha)z^{\alpha}\overline{w}^{\alpha}=\sum\limits_{\alpha\in\mathbb{N}^{m}}\mathbf{e}_{\alpha}(z)\overline{\mathbf{e}_{\alpha}(w)},\quad z,w\in\Omega.$$
So there is $\mathbf{e}_{\alpha}(z)=\sqrt{\widehat{f}(\alpha)}z^{\alpha}$ for all $\alpha\in\mathbb{N}^{m}.$
Since
$$M_{z_{i}}\mathbf{e}_{\alpha}(z)=M_{z_{i}}\sqrt{\widehat{f}(\alpha)}z^{\alpha}=\sqrt{\widehat{f}(\alpha)}z^{\alpha+e_{i}}=\sqrt{\frac{\widehat{f}(\alpha)}{\widehat{f}(\alpha+e_{i})}}\mathbf{e}_{\alpha+e_{i}}(z),\quad 1\leq i\leq m.$$
It following that
$$(M^{*}_{z_{1}}\mathbf{e}_{\alpha+e_{1}}(z),\ldots,M^{*}_{z_{m}}\mathbf{e}_{\alpha+e_{m}}(z))=\left(\sqrt{\frac{\widehat{f}(\alpha)}{\widehat{f}(\alpha+e_{1})}},
\ldots,\sqrt{\frac{\widehat{f}(\alpha)}{\widehat{f}(\alpha+e_{m})}}\right)\mathbf{e}_{\alpha}(z).$$
From $\mathbf{T}$ is unitary equivalent to the adjoint
of $m$-tuple $\mathbf{M}_{z}=(M_{z_{1}},\ldots,M_{z_{m}})$ of multiplication operators on space $\mathcal{H}_{K}$, that is, $\mathbf{T}$ is unitary equivalent to $m$-tuple $M^{*}_{z}=(M^{*}_{z_{1}},\ldots,M^{*}_{z_{m}})$ with weight sequence $\left\{\sqrt{\frac{\widehat{f}(\alpha)}{\widehat{f}(\alpha+e_{1})}},\ldots,\sqrt{\frac{\widehat{f}(\alpha)}{\widehat{f}(\alpha+e_{m})}}\right\}_{\alpha\in\mathbb{N}^{m}}$.
\end{proof}

\begin{lemma}
An $m$-tuple $\mathbf{T}=(T_{1},T_{2},\ldots,T_{m})\in\mathbf{\mathcal{B}}_{n}^{m}(\Omega)$ is reducible, then the Hermitian holomorphic vector bundle $\mathcal{E}_\mathbf{T}$ is reducible.
\end{lemma}
\begin{proof}
If $\mathbf{T}=(T_{1},T_{2},\ldots,T_{m})\in\mathbf{\mathcal{B}}_{n}^{m}(\Omega)$ is reducible, we can let $\mathcal{H}_{1}$ and $\mathcal{H}_{2}$ be the reduced subspace of $\mathbf{T}=(T_{1},T_{2},\ldots,T_{m})$ and satisfy $\mathcal{H}=\mathcal{H}_{1}\oplus\mathcal{H}_{2}$. If $x$ is in $\ker(\mathbf{T}-w)=\bigcap \limits_{i=1}^{m}\ker(T_{i}-w_{i})$, there are $x_{1}$ and $x_{2}$ in $\mathcal{H}_{1}$ and $\mathcal{H}_{2}$, respectively, so that $x=x_{1}\oplus x_{2}$ and $T_{i}x_{1}\oplus T_{i}x_{2}=T_{i}x=w_{i}x=w_{i}x_{1}\oplus w_{i}x_{2}$ for $1\leq i\leq n.$
This means that both $x_{1}$ and $x_{2}$ are in $\ker(\mathbf{T}-w)=\bigcap \limits_{i=1}^{m}\ker(T_{i}-w_{i})$, so $$\ker(\mathbf{T}-w)=\{\ker(\mathbf{T}-w)\cap\mathcal{H}_{1}\}\oplus\{\ker(\mathbf{T}-w)\cap\mathcal{H}_{2}\}.$$
Then $\widetilde{\mathbf{T}}=\mathbf{T}|_{\mathcal{H}_{1}}$ is in $\mathbf{\mathcal{B}}_{n_{1}}^{m}(\Omega)$ and $\widehat{\mathbf{T}}=\mathbf{T}|_{\mathcal{H}_{2}}$ is in $\mathbf{\mathcal{B}}_{n_{2}}^{m}(\Omega)$, where $n=n_{1}+n_{2}$. Therefore, $\mathcal{E}_\mathbf{T}=\mathcal{E}_\mathbf{\widetilde{\mathbf{T}}}\oplus \mathcal{E}_\mathbf{\widehat{\mathbf{T}}}$.
\end{proof}

\begin{corollary}
If $m$-tuples $\mathbf{T}_{i}\sim_{u}(\mathbf{M}^{*}_{z}, \mathcal{H}_{K_{i}}), 1\leq i\leq 3$. Then $K_{1}(z,w)=K_{2}(z,w)\oplus K_{3}(z,w)$ if and only if $\mathbf{T}_{1}\sim_{u}\mathbf{T}_{2}\oplus\mathbf{T}_{3}$.
\end{corollary}

\subsection{The Drury-Arveson space $H_{m}^{2}$}
The Drury-Arveson space $H_{m}^{2}$ is a Hilbert space introduced by Drury in \cite{D1978}, and after Arveson made further research on the space in \cite{A1997}. This space generalizes the classical Hardy space $H^{2}$ on the unit disc to several variables and it is widely used in operator theory and function theory.

The Drury-Arveson space $H_{m}^{2}$ is identified with the space of holomorphic functions $f:\mathbb{B}_{m}\rightarrow \mathbb{C}$ which have a power series $f(z)=\sum\limits_{\alpha\in\mathbb{N}^{m}}a_{\alpha}z^{\alpha}$ such that
$$\Vert f\Vert^{2}_{H_{m}^{2}}:=\sum\limits_{\alpha\in\mathbb{N}^{m}}|a_{\alpha}|^{2}\frac{\alpha!}{|\alpha|!}<\infty.$$
The space $H_{m}^{2}$ is also a reproducing kernel Hilbert space with reproducing kernel
$$K(z,w)=\frac{1}{1-\langle z, w\rangle}=
\sum\limits_{\alpha\in\mathbb{N}^{m}}\frac{|\alpha|!}{\alpha!}z^{\alpha}\overline{w}^{\alpha},\quad z,w\in\mathbb{B}^{m}.$$
For $m$-shift $\mathbf{M}_{z}=(M_{z_{1}},\ldots,M_{z_{m}})$ and $\alpha=(\alpha_{1},\ldots,\alpha_{m})\in\mathbb{N}^{m}$, a small computation reveals that
$$M^{*}_{z_{i}}z^{\alpha}=
\begin{cases}
\frac{\alpha_{i}}{|\alpha|}z^{\alpha-e_{i}}& \text{if}~\alpha_{i}\neq 0 \\
0& \text{otherwise}.
\end{cases}$$
This implies that
$$\sum\limits_{i=1}^{m}M_{z_{i}}M^{*}_{z_{i}}=I-1\otimes1\leq I.$$
Let $\{\mathbf{e}_{\alpha}\}_{\alpha\in\mathbb{N}^{m}}$ be the orthonormal basis of space $H_{m}^{2}$, and then
$$\sum\limits_{\alpha\in\mathbb{N}^{m}}\frac{|\alpha|!}{\alpha!}z^{\alpha}\overline{w}^{\alpha}=K(z,w)=\sum\limits_{\alpha\in\mathbb{N}^{m}}\mathbf{e}_{\alpha}(z)\mathbf{e}^{*}_{\alpha}(w),\quad z,w\in\mathbb{B}^{m}.$$
So $\mathbf{e}_{\alpha}(z)=\sqrt{\frac{|\alpha|!}{\alpha!}}z^{\alpha}$ for all $\alpha\in\mathbb{N}^{m}.$

\begin{lemma}
For the Drury-Arveson space $H_{m}^{2}$ and any $w\in\mathbb{B}^{m}$, we have that
$$\ker(\mathbf{M}^{*}_{z}-w)=\bigcap \limits_{i=1}^{m}\ker(M^{*}_{z_{1}}-w_{i})=\bigvee  K(\cdot, \overline{w}).$$
\end{lemma}
\begin{proof}
Without losing generality, assume that
$$f(z)=\sum\limits_{\alpha\in\mathbb{N}^{m}}a_{\alpha}\mathbf{e}_{\alpha}(z)\in\ker(\mathbf{M}^{*}_{z}-w)=\bigcap \limits_{i=1}^{m}\ker(M^{*}_{z_{i}}-w_{i}).$$
Then we have that
$$\begin{array}{lll}
0&=&(M^{*}_{z_{i}}-w_{i})\sum\limits_{\alpha\in\mathbb{N}^{m}}a_{\alpha}\mathbf{e}_{\alpha}(z)\\[5pt]
&=&M^{*}_{z_{i}}\sum\limits_{\alpha\in\mathbb{N}^{m}}a_{\alpha}\sqrt{\frac{|\alpha|!}{\alpha!}}z^{\alpha}-
\left(w_{i}\sum\limits_{\alpha\in\mathbb{N}^{m}}a_{\alpha}\sqrt{\frac{|\alpha|!}{\alpha!}}z^{\alpha}\right)\\[5pt]
&=&\sum\limits_{\alpha\in\mathbb{N}^{m} \atop \alpha_{i}\geq1}a_{\alpha}\sqrt{\frac{\alpha_{i}}{\vert\alpha\vert}}\sqrt{\frac{|\alpha-e_{i}|!}{(\alpha-e_{i})!}}z^{\alpha-e_{i}}-
\left(w_{i}\sum\limits_{\alpha\in\mathbb{N}^{m}}a_{\alpha}\sqrt{\frac{|\alpha|!}{\alpha!}}z^{\alpha}\right)\\[5pt]
&=&\sum\limits_{\alpha\in\mathbb{N}^{m}}\left[\sqrt{\frac{\alpha_{i}+1}{|\alpha+e_{i}|}}a_{\alpha+e_{i}}-w_{i}a_{\alpha}\right]\mathbf{e}_{\alpha}(z).
\end{array}$$
It follows that $a_{\alpha+e_{i}}=\sqrt{\frac{|\alpha+e_{i}|}{\alpha_{i}+1}}w_{i}a_{\alpha}$ for any $1\leq i\leq m$ and all $\alpha\in\mathbb{N}^{m}.$
There is no loss of generality in assuming $a_{0}=1$, we conclude that $a_{\alpha}=\sqrt{\frac{|\alpha|!}{\alpha!}}w^{\alpha},$ hence that $f(z)=\sum\limits_{\alpha\in\mathbb{N}^{m}}\sqrt{\frac{|\alpha|!}{\alpha!}}w^{\alpha}\mathbf{e}_{\alpha}(z),$ and finally that the result is valid.
\end{proof}

Let $H$ be a separable Hilbert space, $\{\mathbf{e}_{\alpha}\}_{\alpha\in \mathbb{N}^{m}}$ be an orthonormal basis of the Hilbert space $\mathcal{H}=\ell^{2}(\mathbb{N}^{m},H)$ composed of functions $f$ satisfying
$$\Vert f \Vert^{2}=\sum\limits_{\alpha\in\mathbb{N}^{m}}|f(\alpha)|^{2}<\infty.$$
The $m$-tuple $\mathbf{V}=(V_{1},\ldots,V_{m})$ of operator $V_{i}\in\mathcal{L}(\mathcal{H})$ defined by
$$V_{i}\mathbf{e}_{\alpha}=\sqrt{\frac{\alpha_{ i}+1}{|\alpha|+m}}\mathbf{e}_{\alpha+e_{i}}$$
will be referred to as a $\textit{spherical shift}$. For the spherical shift  $\mathbf{V}=(V_{1},\ldots,V_{m})$, one has $V_{i}V_{j}=V_{j}V_{i}$ for $1\leq i, j\leq m$ and also that the equality $I_{\mathcal{H}}-V^{*}_{1}V_{1}-\cdots-V^{*}_{m}V_{m}=0$ holds, thus $V$ is a \textit{spherical isometry}.

\section{The uniqueness of finite strongly irreducible decomposition of operator tuples}

In this section, we mainly use the $K_{0}$-group on the commutative algebra to characterize that the $m$-tuple of commuting operators $\mathbf{T}=(T_{1},\ldots,T_{m})\in\mathcal{L}(\mathcal{H})^{m}$ has unique finite strongly irreducible decomposition up to similarity. In order to prove the main Theorem \ref{2-1}, we need to prove the following Lemmas.

\begin{lemma}\label{2-2}
Let $\mathbf{T}=(T_{1},\ldots,T_{m}), \mathbf{\widetilde{T}}=(\widetilde{T}_{1},\ldots,\widetilde{T}_{m})\in\mathcal{L}(\mathcal{H})^{m}$, and let $\psi$ be an isomorphism from $\mathcal{A}^{\prime}(\mathbf{T})$ to $\mathcal{A}^{\prime}(\mathbf{\widetilde{T}})$. Then $\{P_{i}\}_{i=1}^{n}$ is a unit strongly irreducible decomposition of $\mathbf{T}$ if and only if $\{\psi(P_{i})\}_{i=1}^{n}$ is a strongly irreducible decomposition of $\mathbf{\widetilde{T}}$. In particular, if $\mathbf{T}\sim_{s}\mathbf{\widetilde{T}}$, then $\mathcal{A}^{\prime}(\mathbf{T})\cong\mathcal{A}^{\prime}(\mathbf{\widetilde{T}}).$
\end{lemma}
\begin{proof}
Since $\psi$ is an isomorphic mapping from $\mathcal{A}^{\prime}(\mathbf{T})$ to $\mathcal{A}^{\prime}(\mathbf{\widetilde{T}})$, we only prove the necessity, and its sufficiency can be similarly proved.
Since $\{P_{i}\}_{i=1}^{n}$ is a unit strongly irreducible decomposition of $\mathbf{T}$, we have that $\{P_{i}\}_{i=1}^{n}\subset\mathcal{A}^{\prime}(\mathbf{T})$, and then $\psi(P_{i})\in\mathcal{A}^{\prime}(\mathbf{\widetilde{T}})$, $\psi(P_{i})\psi(P_{j})=\psi(P_{i}P_{j})=\psi(0)=0$ for $i\neq j$, and $\sum\limits_{i=1}^{n}\psi(P_{i})=\psi(\sum\limits_{i=1}^{n}P_{i})=\psi(I)=I$.
From Definition \ref{def1}, $\{\psi(P_{i})\}_{i=1}^{n}$ is a unit decomposition of $\mathbf{\widetilde{T}}$.
Now show we that $\mathbf{\widetilde{T}}|_{\psi(P_{i})\mathcal{H}}, 1\leq i\leq n,$ are strongly irreducible. Assume that there is $i,1\leq i\leq n,$ such that $\mathbf{\widetilde{T}}|_{\psi(P_{i})\mathcal{H}}$ is strongly reducible, that is, there are non-zero idempotents $Q_{1}$ and $Q_{2}$ in $\mathcal{A}^{\prime}(\mathbf{\widetilde{T}})$ such that $Q_{1}+Q_{2}=\psi(P_{i})$ and $Q_{1}Q_{2}=Q_{2}Q_{1}=0$. Since $\psi$ is isomorphic, $\psi^{-1}(Q_{1})$ and $\psi^{-1}(Q_{2})$ are non-zero idempotents in $\mathcal{A}^{\prime}(\mathbf{T})$ and satisfy $P_{i}=\psi^{-1}(Q_{1})+\psi^{-1}(Q_{2})$. Obviously, this contradicts that $\mathbf{T}|_{P_{i}\mathcal{H}}$  is strongly irreducible. So $\{\psi(P_{i})\}_{i=1}^{n}$ is a strongly irreducible decomposition of $\mathbf{\widetilde{T}}$.

If $\mathbf{T}$ is similar to $\mathbf{\widetilde{T}}$, there is an invertible operator $X\in\mathcal{L}(\mathcal{H})$ such that $XT_{i} = \widetilde{T}_{i}X, 1\leq i\leq m$.
Define the mapping $\varphi:\mathcal{A}^{\prime}(\mathbf{T})\rightarrow\mathcal{A}^{\prime}(\mathbf{\widetilde{T}})$ as
$\varphi(Y)=XYX^{-1}$ for any $Y\in\mathcal{A}^{\prime}(\mathbf{T}).$
Then $\varphi$ is an isomorphic mapping and $\mathcal{A}^{\prime}(\mathbf{T})\cong\mathcal{A}^{\prime}(\mathbf{\widetilde{T}}).$
\end{proof}

\begin{lemma}\label{2-3}
Let $\mathbf{T}=(T_{1},\ldots,T_{m})\in\mathcal{L}(\mathcal{H})^{m}$, and let $P_{1}$ and $P_{2}$ be idempotent operators in $\mathcal{A}^{\prime}(\mathbf{T})$.
If $P_{1}\sim_{s}(\mathcal{A}^{\prime}(\mathbf{T}))P_{2}$, then $\mathbf{T}|_{P_{1}\mathcal{H}}\sim_{s}\mathbf{T}|_{P_{2}\mathcal{H}}$, that is, there is an invertible operator $Y\in\mathcal{L}(P_{1}\mathcal{H},P_{2}\mathcal{H})$ such that
$YT_{i}|_{P_{1}\mathcal{H}} = T_{i}|_{P_{2}\mathcal{H}}Y$ for $1\leq i\leq m.$
\end{lemma}
\begin{proof}
Since $P_{1}\sim_{s}(\mathcal{A}^{\prime}(\mathbf{T}))P_{2}$, there is an operator $X\in GL(\mathcal{A}^{\prime}(\mathbf{T}))$ such that $XP_{1}X^{-1}=P_{2}$, and then $X(I-P_{1})X^{-1}=(I-P_{2})$.
It follows that $X \text{ran\,}P_{1}=\text{ran\,}P_{2}$ and $X\text{ran\,}(I-P_{1})=\text{ran\,} (I-P_{2})$.
Letting $X_{1}:=X|_{\text{ran\,}P_{1}}$ and $X_{2}:=X|_{\text{ran\,}(I-P_{1})}$, we have that $X_{1}\in GL(\mathcal{L}(P_{1}\mathcal{H},P_{2}\mathcal{H}))$,
$X_{2}\in GL(\mathcal{L}((I-P_{1})\mathcal{H},(I-P_{2})\mathcal{H}))$ and $X=X_{1}\dot{+}X_{2}$, where $\dot{+}$ denotes the topological direct sum.
Since $P_{1}$ and $P_{2}$ are idempotent operators in $\mathcal{A}^{\prime}(\mathbf{T})$, we obtain that
\begin{equation}
\begin{array}{lll}
\mathbf{T}=(T_{1},T_{2},\ldots,T_{m})&=\left(\begin{pmatrix}T_{0,1}&0\\ 0&T_{1,1}\end{pmatrix},\begin{pmatrix}T_{0,2}&0\\ 0&T_{1,2}\end{pmatrix},\cdots,\begin{pmatrix}T_{0,m}&0\\ 0&T_{1,m}\end{pmatrix}\right)\begin{matrix}P_{1}\mathcal{H}\\ (I-P_{1})\mathcal{H}\end{matrix}\\[10pt]
&=\left(\begin{pmatrix}\widetilde{T}_{0,1}&0\\ 0&\widetilde{T}_{1,1}\end{pmatrix},\begin{pmatrix}\widetilde{T}_{0,2}&0\\ 0&\widetilde{T}_{1,2}\end{pmatrix},\cdots,\begin{pmatrix}\widetilde{T}_{0,m}&0\\ 0&\widetilde{T}_{1,m}\end{pmatrix}\right)\begin{matrix}P_{2}\mathcal{H}\\ (I-P_{2})\mathcal{H}\end{matrix},\notag
\end{array}
\end{equation}
where $T_{0,i}=T_{i}|_{P_{1}\mathcal{H}}$, $T_{1,i}=T_{i}|_{(I-P_{1})\mathcal{H}}$, $\widetilde{T}_{0,i}=T_{i}|_{P_{2}\mathcal{H}}$ and $\widetilde{T}_{1,i}=T_{i}|_{(I-P_{2})\mathcal{H}}$ for $1\leq i\leq m.$ From $XP_{1}X^{-1}=P_{2}$, $X(I-P_{1})X^{-1}=(I-P_{2})$, $X\in GL(\mathcal{A}^{\prime}(\mathbf{T}))$ and $P_{1}\in\mathcal{A}^{\prime}(\mathbf{T})$, we know that $$T_{i}P_{2}X=T_{i}XP_{1}=XT_{i}P_{1}\quad\text{and}\quad T_{i}(I-P_{2})X=T_{i}X(I-P_{1})=XT_{i}(I-P_{1}), \quad 1\leq i\leq m.$$
This means that
$$\left(\begin{pmatrix}\widetilde{T}_{0,1}&0\\ 0&\widetilde{T}_{1,1}\end{pmatrix},\cdots,\begin{pmatrix}\widetilde{T}_{0,m}&0\\ 0&\widetilde{T}_{1,m}\end{pmatrix}\right)\begin{pmatrix}X_{1}&0\\ 0&X_{2}\end{pmatrix}=\begin{pmatrix}X_{1}&0\\ 0&X_{2}\end{pmatrix}\left(\begin{pmatrix}T_{0,1}&0\\ 0&T_{1,1}\end{pmatrix},\cdots,\begin{pmatrix}T_{0,m}&0\\ 0&T_{1,m}\end{pmatrix}\right).$$
Thus $\mathbf{T}|_{P_{1}\mathcal{H}}=(T_{0,1},\ldots,T_{0,m})$ is similar to $\mathbf{T}|_{P_{2}\mathcal{H}}=(\widetilde{T}_{0,1},\ldots,\widetilde{T}_{0,m}).$
\end{proof}

\begin{lemma}\label{2-4}
Let $\mathbf{T}=(T_{1},\cdots,T_{m})\in\mathcal{L}(\mathcal{H})^{m}$. Suppose that $\{P_{i}\}_{i=1}^{n}$ and $\{Q_{i}\}_{i=1}^{n}$ are two unit strongly irreducible decomposition of $\mathbf{T}$. If there exist $X_{i}\in GL(\mathcal{L}(P_{i}\mathcal{H},Q_{i}\mathcal{H}))$ such that
$$X_{i}(T_{j}|_{P_{i}\mathcal{H}})X_{i}^{-1}=T_{j}|_{Q_{i}\mathcal{H}},\quad 1\leq i\leq n, 1\leq j\leq m.$$
Then $X=X_{1}\dot{+}X_{2}\dot{+}\cdots\dot{+}X_{n}\in GL(\mathcal{A}^{\prime}(\mathbf{T}))$, where $\dot{+}$ denotes the topological direct sum.
\end{lemma}
\begin{proof}
Since $\{P_{i}\}_{i=1}^{n}$ and $\{Q_{i}\}_{i=1}^{n}$ are two unit strongly irreducible decomposition of $\mathbf{T}$, we have that
$$\text{ran\,} P_{1}\dot{+}\text{ran\,}P_{2}\dot{+}\cdots\dot{+}\text{ran\,}P_{n}=\mathcal{H}=\text{ran\,} Q_{1}\dot{+}\text{ran\,}Q_{2}\dot{+}\cdots\dot{+}\text{ran\,}Q_{n}.$$
Therefore,
\begin{equation}
\begin{array}{lll}
\mathbf{T}&=\left(\begin{pmatrix}\begin{matrix}T_{1,1}&&&&\\ &T_{1,2}&&&\text{{\Large 0}}&\\&&T_{1,3}&&\\&&&\ddots&\\&\text{{\Large 0}}&&&T_{1,n}\end{matrix}\end{pmatrix},
\cdots,\begin{pmatrix}\begin{matrix}T_{m,1}&&&&\\ &T_{m,2}&&&\text{{\Large 0}}&\\&&T_{m,3}&&\\&&&\ddots&\\&\text{{\Large 0}}&&&T_{m,n}\end{matrix}\end{pmatrix}\right)\begin{matrix}P_{1}\mathcal{H}\\ P_{2}\mathcal{H}\\
P_{3}\mathcal{H}\\ \vdots\\ P_{n}\mathcal{H}\end{matrix}\\[40pt]
&=\left(\begin{pmatrix}\begin{matrix}\widetilde{T}_{1,1}&&&&\\ &\widetilde{T}_{1,2}&&&\text{{\Large 0}}&\\&&\widetilde{T}_{1,3}&&\\&&&\ddots&\\&\text{{\Large 0}}&&&\widetilde{T}_{1,n}\end{matrix}\end{pmatrix},
\cdots,\begin{pmatrix}\begin{matrix}\widetilde{T}_{m,1}&&&&\\ &\widetilde{T}_{m,2}&&&\text{{\Large 0}}&\\&&\widetilde{T}_{m,3}&&\\&&&\ddots&\\&\text{{\Large 0}}&&&\widetilde{T}_{m,n}\end{matrix}\end{pmatrix}\right)\begin{matrix}Q_{1}\mathcal{H}\\ Q_{2}\mathcal{H}\\
Q_{3}\mathcal{H}\\ \vdots\\ Q_{n}\mathcal{H}\end{matrix},\notag
\end{array}
\end{equation}
where $T_{j,i}=T_{j}|_{P_{i}\mathcal{H}}$ and $\widetilde{T}_{j,i}=T_{j}|_{Q_{i}\mathcal{H}}$ for $1\leq i\leq n$ and $1\leq j\leq m.$ Since
$$X_{i}(T_{j}|_{P_{i}\mathcal{H}})X_{i}^{-1}=T_{j}|_{Q_{i}\mathcal{H}},\quad 1\leq i\leq n, \,1\leq j\leq m,$$
we have that
\begin{equation}
\begin{array}{lll} \quad\begin{pmatrix}\begin{smallmatrix}X_{1}&&&&\\ &X_{2}&&&\text{{\Large 0}}&\\&&X_{3}&&\\&&&\ddots&\\&\text{{\Large 0}}&&&X_{n}\end{smallmatrix}\end{pmatrix} \begin{pmatrix}\begin{smallmatrix}T_{j,1}&&&&\\ &T_{j,2}&&&\text{{\Large 0}}&\\&&T_{j,3}&&\\&&&\ddots&\\&\text{{\Large 0}}&&&T_{j,n}\end{smallmatrix}\end{pmatrix} =\begin{pmatrix}\begin{smallmatrix}\widetilde{T}_{j,1}&&&&\\ &\widetilde{T}_{j,2}&&&\text{{\Large 0}}&\\&&\widetilde{T}_{j,3}&&\\&&&\ddots&\\&\text{{\Large 0}}&&&\widetilde{T}_{j,n}\end{smallmatrix}\end{pmatrix}\begin{pmatrix}\begin{smallmatrix}X_{1}&&&&\\ &X_{2}&&&\text{{\Large 0}}&\\&&X_{3}&&\\&&&\ddots&\\&\text{{\Large 0}}&&&X_{n}\end{smallmatrix}\end{pmatrix}\notag \end{array} \end{equation}
for $1\leq j\leq m.$ Setting $X=X_{1}\dot{+}X_{2}\dot{+}\cdots\dot{+}X_{n}$, where $\dot{+}$ denotes the topological direct sum.
then $X(T_{1},\ldots,T_{m})=(T_{1},\ldots,T_{m})X$ and $X$ is invertible. Thus $X=X_{1}\dot{+}X_{2}\dot{+}\cdots\dot{+}X_{n}\in GL(\mathcal{A}^{\prime}(\mathbf{T})).$
\end{proof}

\begin{lemma}\label{2-5}
Let $\mathbf{T}=(T_{1},\ldots,T_{m})\in\mathcal{L}(\mathcal{H})^{m}$, and let
$$\{P_{1},\ldots,P_{k},P_{k+1},\ldots,P_{n}\}\quad\text{and}\quad \{Q_{1},\ldots,Q_{k},Q_{k+1},\ldots,Q_{n}\}$$
be two sets of idempotent operators in $\mathcal{A}^{\prime}(\mathbf{T})$. If there are $X,Y\in GL(\mathcal{A}^{\prime}(\mathbf{T}))$ and a permutation $\Pi\in S_{n}$ satisfying
\begin{itemize}
  \item [(1)]$XP_{i}X^{-1}=Q_{i},\,\, 1\leq i\leq k;$
  \item [(2)]$Y^{-1}P_{i}Y=Q_{\Pi(i)},\,\, 1\leq i\leq n$.
\end{itemize}
Then for any $Q_{r}\in\{Q_{i}\}_{i=k+1}^{n}$, there is $P_{r^{\prime}}\in\{P_{i}\}_{i=k+1}^{n}$ and invertible operator $Z_{r}$, a finite product of $X$ and $Y$, such that $$Z_{r}Q_{r}Z_{r}^{-1}=P_{r^{\prime}}.$$
More specifically, $\{P_{(k+1)^{\prime}},P_{(k+2)^{\prime}},\ldots,P_{n^{\prime}}\}$ is a rearrangement of $\{P_{i}\}_{i=k+1}^{n}$.
\end{lemma}
\begin{proof}
For any $Q_{r}\in\{Q_{i}\}_{i=k+1}^{n}$, it can be seen from Property $(2)$ that there is $P_{j_{1}}\in\{P_{i}\}_{i=1}^{n}$ such that
\begin{equation}\label{2.1}
P_{j_{1}}=YQ_{r}Y^{-1}.
\end{equation}
If $k< j_{1}\leq n $, then set $Z_{r}=Y$ and $P_{r^{\prime}}=P_{j_{1}}$. If $1\leq j_{1}\leq k,$ then $Q_{j_{1}}=XP_{j_{1}}X^{-1}=XYQ_{r}Y^{-1}X^{-1}$ can be obtained from Property $(1)$.
By Property $(2)$, there is $P_{j_{2}}\in\{P_{i}\}_{i=1}^{n}$ satisfying
\begin{equation}\label{2.2}
P_{j_{2}}=YQ_{j_{1}}Y^{-1}=YXYQ_{r}Y^{-1}X^{-1}Y^{-1}.
\end{equation}
It is clear that $j_{1}\neq j_{2}$. Otherwise, from (\ref{2.1}) and (\ref{2.2}),
it is known that
$$Q_{j_{1}}=Y^{-1}P_{j_{2}}Y=Y^{-1}P_{j_{1}}Y=Q_{r},$$  which is obviously contradictory to
$1\leq j_{1}\leq k$ and $k+1\leq r\leq n$.
If $k< j_{2}\leq n $, then set $Z_{r}=YXY$ and $P_{r^{\prime}}=P_{j_{2}}$. If $1\leq j_{2}\leq k,$
by using Properties $(1),(2)$ and (\ref{2.2}), there exists $P_{j_{3}}\in\{P_{i}\}_{i=1}^{n}$ such that
\begin{equation}\label{2.3}
P_{j_{3}}=YQ_{j_{2}}Y^{-1}=YXP_{j_{2}}X^{-1}Y^{-1}=YXYXYQ_{r}Y^{-1}X^{-1}Y^{-1}X^{-1}Y^{-1}.
\end{equation}
Similarly, $j_{3}\notin\{j_{1},j_{2}\}$. Otherwise, if $j_{3}=j_{1}$, from (\ref{2.1}) and (\ref{2.3}), we have that
$$Q_{j_{2}}=Y^{-1}P_{j_{3}}Y=Y^{-1}P_{j_{1}}Y=Q_{r}.$$
If $j_{3}=j_{2}$, we can know $Q_{j_{2}}=Y^{-1}P_{j_{3}}Y=Y^{-1}P_{j_{2}}Y=Q_{j_{1}}$ from (\ref{2.2}) and (\ref{2.3}). Obviously, these are contradictory.
If $k< j_{3}\leq n $, then set $Z_{r}=YXYXY$ and $P_{r^{\prime}}=P_{j_{3}}$. Otherwise, we can continue the above choice process. Since $n$ is a natural number, after $t$ steps, $t\leq k+1$, we must be able to find $P_{j_{t}}\in\{P_{i}\}_{i=k+1}^{n}.$ Setting
$$P_{r^{\prime}}=P_{j_{t}}\quad\text{and}\quad Z_{r}=YXY\cdots XY\quad \text{(X appears}\,\,t-1\,\,\text{times),}$$
then $Z_{r}Q_{r}Z_{r}^{-1}=P_{j_{t}}$.
Assert that $\{P_{(k+1)^{\prime}},P_{(k+2)^{\prime}},\ldots,P_{n^{\prime}}\}$ is a rearrangement of $\{P_{i}\}_{i=k+1}^{n}$, that is, if $r_{1}\neq r_{2}$, $k<r_{1},r_{2}\leq n$, then $j_{t_{1}}\neq j_{t_{2}},k<j_{t_{1}},j_{t_{2}}\leq n$. Otherwise, according to the above choice process, there are
$$Z_{r_{1}}=YXY\cdots XY \,\,\text{(X appears} \,\,t_{1}-1 \,\,\text{times})\quad\text{and}\quad Z_{r_{2}}=YXY\cdots XY \,\,\text{(X appears} \,\,t_{2}-1 \,\,\text{times}),$$
such that
$$Z_{r_{1}}Q_{r_{1}}Z_{r_{1}}^{-1}=P_{j_{t_{1}}}=P_{j_{t_{2}}}=Z_{r_{2}}Q_{r_{2}}Z_{r_{2}}^{-1}.$$ Without losing generality, assume that $r_{1}> r_{2}$,
then $Q_{r_{2}}=Z_{r_{2}}^{-1}Z_{r_{1}}Q_{r_{1}}Z_{r_{1}}^{-1}Z_{r_{2}}\in\{Q_{i}\}_{i=k+1}^{n},$
where $Z_{r_{2}}^{-1}Z_{r_{1}}=XY\cdots XY \,\,\text{(X appears} \,\,t_{1}-t_{2} \,\,\text{times}).$ Setting
$$R:=YXY\cdots XY \quad \text{(X appears} \,\,t_{1}-t_{2}-1 \,\,\text{times}).$$
By this choice process, we know that $RQ_{r_{1}}R^{-1}\in\{P_{i}\}_{i=1}^{k}.$ From Property $(1)$, we get $XRQ_{r_{1}}R^{-1}X^{-1}\in\{Q_{i}\}_{i=1}^{k}$. But
$$XRQ_{r_{1}}R^{-1}X^{-1}=Z_{r_{2}}^{-1}Z_{r_{1}}Q_{r_{1}}Z_{r_{1}}^{-1}Z_{r_{2}}=Q_{r_{2}}\in\{Q_{i}\}_{i=k+1}^{n},$$
obviously, this is contradictory. Similarly, we also know that if $r_{1}<r_{2}$, there is $j_{t_{1}}\neq j_{t_{2}}$.
If $r_{1}=r_{2}$, through this choice process, we have that $j_{t_{1}}= j_{t_{2}}.$ This completes the proof of this Lemma.
\end{proof}

According to the above Lemma, we can get the following Lemma, and its proof process is similar to the above Lemma proof process, so we don't write its detailed proof process.

\begin{lemma}\label{2-6}
Let $\mathbf{T}=(T_{1},\ldots,T_{m})\in\mathcal{L}(\mathcal{H})^{m}$,
$$\{P_{1},\ldots,P_{k_{1}},\ldots,P_{k_{s}},P_{k_{s}+1}\ldots,P_{n}\}\quad\text{and}\quad\{Q_{1},\ldots,Q_{k_{1}},\ldots,Q_{k_{s}},Q_{k_{s}+1}\ldots,Q_{n}\}$$ be two sets of idempotent operators in $\mathcal{A}^{\prime}(\mathbf{T})$. If there exist $\{X_{i}\}_{i=1}^{s},Y\in GL(\mathcal{A}^{\prime}(\mathbf{T}))$ and a permutation $\Pi\in \mathcal{S}_{n}$ satisfying
\begin{itemize}
  \item [(1)]$X_{i}P_{j}X_{i}^{-1}=Q_{j},\,\, k_{i}+1\leq j\leq k_{i+1}, k_{0}=0, 0\leq i\leq s-1;$
  \item [(2)]$Y^{-1}P_{i}Y=Q_{\Pi(i)},\,\, 1\leq i\leq n$.
\end{itemize}
Then for any $Q_{r}\in\{Q_{i}\}_{i=k_{s}+1}^{n}$, there is $P_{r^{\prime}}\in\{P_{i}\}_{i=k_{s}+1}^{n}$ and invertible operator $Z_{r}$, a finite product of $\{X_{i}\}_{i=1}^{s}$ and $Y$, such that
$$Z_{r}Q_{r}Z_{r}^{-1}=P_{r^{\prime}}.$$
More specifically, $\{P_{(k_{s}+1)^{\prime}},P_{(k_{s}+2)^{\prime}},\ldots,P_{n^{\prime}}\}$ is a rearrangement of $\{P_{i}\}_{i=k_{s}+1}^{n}$.
\end{lemma}

\begin{lemma}\label{2-7}
Let $\mathbf{T}=(T_{1},\ldots,T_{m})\in\mathcal{L}(\mathcal{H})^{m}$,
$$\{P_{1},\ldots,P_{k},P_{k+1},\ldots,P_{n}\}\quad \text{and}\quad \{Q_{1},\ldots,Q_{k},Q_{k+1},\ldots,Q_{n}\}$$
be two
unit decompositions of $\mathbf{T}$. If the following properties are satisfied.
\begin{itemize}
  \item [(1)]There exists an $X_{j}\in GL(\mathcal{L}(P_{j}\mathcal{H},Q_{j}\mathcal{H}))$ such that
  $$X_{j}(T_{i}|_{P_{j}\mathcal{H}})X_{j}^{-1} = T_{i}|_{Q_{j}\mathcal{H}},\quad 1\leq j\leq k,1\leq i\leq m.$$
  \item [(2)]There are $Y\in GL(\mathcal{A}^{\prime}(\mathbf{T}))$ and a permutation $\Pi\in \mathcal{S}_{n}$ satisfying $$Y^{-1}P_{j}Y=Q_{\Pi(j)},\quad 1\leq j\leq n.$$
\end{itemize}
Then for any $Q_{r}\in\{Q_{i}\}_{i=k+1}^{n}$, there is $P_{r^{\prime}}\in\{P_{i}\}_{i=k+1}^{n}$ and $Z_{r}\in GL(\mathcal{L}(Q_{r}\mathcal{H},P_{r^{\prime}}\mathcal{H}))$ such that $$Z_{r}(T_{i}|_{Q_{r}\mathcal{H}})Z_{r}^{-1} = T_{i}|_{P_{r^{\prime}}\mathcal{H}},\quad 1\leq i\leq m.$$
Specifically, if $r_{1}\neq r_{2}$, then $r_{1}^{\prime}\neq r_{2}^{\prime}$.
\end{lemma}
\begin{proof}
For any $Q_{r}\in\{Q_{i}\}_{i=k+1}^{n}$, from Property $(2)$, there exists $P_{j_{1}}\in\{P_{i}\}_{i=1}^{n}$ such that
\begin{equation}\label{3.1}
P_{j_{1}}=YQ_{r}Y^{-1}.
\end{equation}
If $k< j_{1}\leq n $, setting $Z_{r}:=Y|_{Q_{r}\mathcal{H}}$ and $P_{r^{\prime}}:=P_{j_{1}}$, then
$Z_{r}(T_{i}|_{Q_{r}\mathcal{H}})Z_{r}^{-1} = T_{i}|_{P_{r^{\prime}}\mathcal{H}}$ for $1\leq i\leq m.$ Otherwise, if $1\leq j_{1}\leq k,$ from Property $(1)$, $$X_{j_{1}}T_{i}|_{(YQ_{r}Y^{-1})\mathcal{H}}X_{j_{1}}^{-1}=X_{j_{1}}T_{i}|_{P_{j_{1}}\mathcal{H}}X_{j_{1}}^{-1}=T_{i}|_{Q_{j_{1}}\mathcal{H}},\quad 1\leq i\leq m.$$
By Property $(2)$, there is $P_{j_{2}}\in\{P_{i}\}_{i=1}^{n}$ satisfying
\begin{equation}\label{3.2}
P_{j_{2}}=YQ_{j_{1}}Y^{-1}.
\end{equation}
It is clear that $j_{1}\neq j_{2}$. Otherwise, from (\ref{3.1}) and (\ref{3.2}),
$Q_{j_{1}}=Y^{-1}P_{j_{2}}Y=Y^{-1}P_{j_{1}}Y=Q_{r}$, which is contradictory to
$1\leq j_{1}\leq k$ and $k+1\leq r\leq n$.
If $k< j_{2}\leq n $, setting $$Z_{r}:=Y|_{Q_{j_{1}}\mathcal{H}}X_{j_{1}}Y|_{Q_{r}\mathcal{H}}\quad \text{and}\quad P_{r^{\prime}}:=P_{j_{2}},$$
then $Z_{r}(T_{i}|_{Q_{r}\mathcal{H}})Z_{r}^{-1} = T_{i}|_{P_{r^{\prime}}\mathcal{H}}$ for $1\leq i\leq m.$ If $1\leq j_{2}\leq k,$
by Properties $(1)$, we have that
$$X_{j_{2}}T_{i}|_{P_{j_{2}}\mathcal{H}}X_{j_{2}}^{-1}=T_{i}|_{Q_{j_{2}}\mathcal{H}},\quad 1\leq i\leq m.$$
Using Properties $(2)$ again, there is $P_{j_{3}}\in\{P_{i}\}_{i=1}^{n}$ such that
\begin{equation}\label{3.3}
P_{j_{3}}=YQ_{j_{2}}Y^{-1}.
\end{equation}
Similarly, $j_{3}\notin\{j_{1},j_{2}\}$. Otherwise, if $j_{3}=j_{1}$, we know that $Q_{j_{2}}=Y^{-1}P_{j_{3}}Y=Y^{-1}P_{j_{1}}Y=Q_{r}$ from (\ref{3.1}) and (\ref{3.3}),
if $j_{3}=j_{2}$, we know that $Q_{j_{2}}=Y^{-1}P_{j_{3}}Y=Y^{-1}P_{j_{2}}Y=Q_{j_{1}}$ from (\ref{3.2}) and (\ref{3.3}). Obviously, these are contradictory.
If $k< j_{3}\leq n $, then set
$$Z_{r}:=Y|_{Q_{j_{2}}\mathcal{H}}X_{j_{2}}Y|_{Q_{j_{1}}\mathcal{H}}X_{j_{1}}Y|_{Q_{r}\mathcal{H}}\quad\text{and}\quad P_{r^{\prime}}:=P_{j_{3}}.$$ Otherwise, we can continue the above choice process. Since $n$ is a natural number, after $t$ steps, $t\leq k+1$, we can find $P_{j_{t}}\in\{P_{i}\}_{i=k+1}^{n}.$ Setting
$$P_{r^{\prime}}:=P_{j_{t}}\quad\text{and}\quad Z_{r}:=Y|_{Q_{j_{t-1}}\mathcal{H}}X_{j_{t-1}}Y|_{Q_{j_{t-2}}\mathcal{H}}X_{j_{t-2}}\cdots Y|_{Q_{j_{1}}\mathcal{H}}X_{j_{1}}Y|_{Q_{r}\mathcal{H}}\quad \text{(Y appears}\,\,t\,\,\text{times),}$$
then $Z_{r}T_{i}|_{Q_{r}\mathcal{H}}Z_{r}^{-1}=T_{i}|_{P_{j_{t}}\mathcal{H}}$ for $1\leq i\leq m$.

Now let's prove that if $r_{1}\neq r_{2}$, $k<r_{1},r_{2}\leq n$, then $j_{t_{1}}\neq j_{t_{2}},k<j_{t_{1}},j_{t_{2}}\leq n$.

Through the proof process above. Firstly, we assume that there are
$r_{1}\neq r_{2}$, $k<r_{1}, r_{2}\leq n$, such that $j_{1}= j_{2},k<j_{1},j_{2}\leq n$, and
$$Z_{r_{1}}(T_{i}|_{Q_{r_{1}}\mathcal{H}})Z_{r_{1}}^{-1}=T_{i}|_{P_{j_{1}\mathcal{H}}}=T_{i}|_{P_{j_{2}\mathcal{H}}}=Z_{r_{2}}(T_{i}|_{Q_{r_{2}}\mathcal{H}})Z_{r_{2}}^{-1},\quad 1\leq i\leq m,$$
where $Z_{r_{1}}= Y|_{Q_{r_{1}}\mathcal{H}}, Z_{r_{2}}=Y|_{Q_{r_{2}}\mathcal{H}}, P_{j_{1}}=YQ_{r_{1}}Y^{-1},$ and $P_{j_{2}}=YQ_{r_{2}}Y^{-1}$
Thus,
\begin{equation}\label{eq1}
T_{i}Q_{r_{2}}=T_{i}|_{Q_{r_{2}}\mathcal{H}}=Z_{r_{2}}^{-1}T_{i}|_{P_{j_{1}\mathcal{H}}}Z_{r_{2}}
=\left(Y|_{Q_{r_{2}}\mathcal{H}}\right)^{-1}T_{i}|_{P_{j_{1}\mathcal{H}}}Y|_{Q_{r_{2}}\mathcal{H}}
=Y^{-1}|_{P_{j_{2}}\mathcal{H}}T_{i}|_{P_{j_{1}\mathcal{H}}}Y|_{Q_{r_{2}}\mathcal{H}}.
\end{equation}
Since $\{P_{i}\}_{i=1}^{n}$ is a unit decomposition of $\mathbf{T}$,
(\ref{eq1}) shows that $j_{1}= j_{2}$ by $P_{j_{1}}P_{j_{2}}=
\left\{
\begin{array}{lc}
1 & j_{1}=j_{2}\\
0 & \text{else}\\
\end{array}
\right.$.
There is no loss of generality in assuming that there are $r_{1}\neq r_{2}$, $k<r_{1}, r_{2}\leq n$, such that $j_{t_{1}}= j_{t_{2}},k<j_{t_{1}},j_{t_{2}}\leq n$, and then
$$Z_{r_{1}}(T_{i}|_{Q_{r_{1}}\mathcal{H}})Z_{r_{1}}^{-1}=T_{i}|_{P_{j_{t_{1}}\mathcal{H}}}=T_{i}|_{P_{j_{t_{2}}\mathcal{H}}}=Z_{r_{2}}(T_{i}|_{Q_{r_{2}}\mathcal{H}})Z_{r_{2}}^{-1},\quad 1\leq i\leq m,$$
where $$Z_{r_{1}}=Y|_{Q_{j_{t-1}}\mathcal{H}}X_{j_{t-1}}Y|_{Q_{j_{t-2}}\mathcal{H}}X_{j_{t-2}}\cdots Y|_{Q_{j_{l}}\mathcal{H}}X_{j_{l}} Y|_{Q_{j_{l-1}}\mathcal{H}}X_{j_{l-1}} \cdots Y|_{Q_{j_{1}}\mathcal{H}}X_{j_{1}}Y|_{Q_{r_{1}}\mathcal{H}},$$
$$Z_{r_{2}}=Y|_{Q_{j_{t-1}}\mathcal{H}}X_{j_{t-1}}Y|_{Q_{j_{t-2}}\mathcal{H}}X_{j_{t-2}}\cdots Y|_{Q_{j_{l}}\mathcal{H}}X_{j_{l}}Y|_{Q_{r_{2}}\mathcal{H}}\quad
\text{and}\quad Y|_{Q_{r_{2}}\mathcal{H}}=Y|_{Q_{j_{l-1}}}.$$
Therefore,
$$\begin{array}{lll}
&&T_{i}|_{Q_{r_{2}}\mathcal{H}}\\[5pt]
&=&Z_{r_{2}}^{-1}Z_{r_{1}}(T_{i}|_{Q_{r_{1}}\mathcal{H}})Z_{r_{1}}^{-1}Z_{r_{2}}\\[5pt]
&=&(X_{j_{l-1}} Y|_{Q_{j_{l-2}}\mathcal{H}}\cdots Y|_{Q_{j_{1}}\mathcal{H}}X_{j_{1}}Y|_{Q_{r_{1}}\mathcal{H}})(T_{i}|_{Q_{r_{1}}\mathcal{H}})(X_{j_{l-1}}Y|_{Q_{j_{l-2}}\mathcal{H}}\cdots Y|_{Q_{j_{1}}\mathcal{H}}X_{j_{1}}Y|_{Q_{r_{1}}\mathcal{H}})^{-1}.
\end{array}$$
Setting $Z_{r}:= Y|_{Q_{j_{l-2}}\mathcal{H}}\cdots Y|_{Q_{j_{1}}\mathcal{H}}X_{j_{1}}Y|_{Q_{r_{1}}\mathcal{H}}.$
With the above choice process, we know that
$$Z_{r}(T_{i}|_{Q_{r_{1}}\mathcal{H}})Z_{r}^{-1}=T_{i}|_{P_{j_{l-1}}\mathcal{H}},\quad 1\leq i\leq m,$$
where $j_{l-1}\in\{1,2,\ldots,k\}$.
From Property $(1)$, we have that
$$T_{i}|_{Q_{r_{2}}\mathcal{H}}=Z_{r_{2}}^{-1}Z_{r_{1}}(T_{i}|_{Q_{r_{1}}\mathcal{H}})Z_{r_{1}}^{-1}Z_{r_{2}}=
X_{j_{l-1}}(T_{i}|_{P_{j_{l-1}}\mathcal{H}})(X_{j_{l-1}})^{-1}=T_{i}|_{Q_{j_{l-1}}\mathcal{H}},\quad 1\leq i\leq m.$$
Then $r_{2}=j_{l-1},$
but $1\leq j_{l-1}\leq k< r_{2}\leq n$, this is impossible, so $j_{t_{1}}\neq j_{t_{2}}$ if $r_{1}\neq r_{2}$.
\end{proof}

\begin{lemma}\label{2-8}
Let $\mathbf{T}=(T_{1},\ldots,T_{m})\in\mathcal{L}(\mathcal{H})^{m}$ has unique finite strongly irreducible decomposition up to similarity. Then for
any idempotent operator $P\in\mathcal{A}^{\prime}(\mathbf{T})$, $\mathbf{T}|_{P\mathcal{H}}=(T_{1}|_{P\mathcal{H}},\ldots,T_{m}|_{P\mathcal{H}})$ has unique finite strongly irreducible decomposition up to similarity.
\end{lemma}
\begin{proof}
Since $\mathbf{T}=(T_{1},\ldots,T_{m})\in\mathcal{L}(\mathcal{H})^{m}$ has unique finite strongly irreducible decomposition up to similarity, we know that $\mathbf{T}|_{P\mathcal{H}}=(T_{1}|_{P\mathcal{H}},\ldots,T_{m}|_{P\mathcal{H}})$ has finite strongly irreducible decomposition.
Let $\{P_{i}\}_{i=1}^{k}$ and $\{Q_{i}\}_{i=1}^{\widetilde{k}}$ be two unit finite strongly irreducible decomposition of $\mathbf{T}|_{P\mathcal{H}}$, and let $\{P_{i}\}_{i=k+1}^{n}$ be a unit finite strongly irreducible decomposition of $\mathbf{T}|_{(I-P)\mathcal{H}}$. Then $\{\{P_{i}\}_{i=1}^{k},\{P_{i}\}_{i=k+1}^{n}\}$ and $\{\{Q_{i}\}_{i=1}^{\widetilde{k}},\{P_{i}\}_{i=k+1}^{n}\}$ are two unit finite strongly irreducible decomposition of $\mathbf{T}$. From the uniqueness of strongly irreducible decomposition of $\mathbf{T}$, we know that $k=\widetilde{k}$ and there is an operator $Y\in GL(\mathcal{A}^{\prime}(\mathbf{T}))$ such that $$\{Y^{-1}P_{i}Y\}_{i=1}^{n}=\{\{Q_{i}\}_{i=1}^{k},\{P_{i}\}_{i=k+1}^{n}\}.$$
Therefore, for any $P_{i}\in\{P_{i}\}_{i=k+1}^{n}$, there is $I|_{P_{i}\mathcal{H}}\in GL(\mathcal{L}(P_{i}\mathcal{H},P_{i}\mathcal{H}))$ satisfies
$$I|_{P_{i}\mathcal{H}}(T_{j}|_{P_{i}\mathcal{H}})I|_{P_{i}\mathcal{H}}^{-1} = T_{j}|_{P_{i}\mathcal{H}},\quad 1\leq j\leq m, k+1\leq i\leq n.$$
And there are $Y\in GL(\mathcal{A}^{\prime}(\mathbf{T}))$ and a permutation $\Pi\in \mathcal{S}_{n}$ such that $$Y^{-1}P_{i}Y=Q_{\Pi(i)},\quad 1\leq i\leq n,$$ where $Q_{j}=P_{j}$ for $k+1\leq j\leq n$.
By Lemma \ref{2-7}, there is a permutation $\widetilde{\Pi}\in \mathcal{S}_{k}$ and $Z_{i}\in GL(\mathcal{L}(Q_{i}\mathcal{H},P_{\widetilde{\Pi}(i)}\mathcal{H})), 1\leq i\leq k,$ such that
\begin{equation}\label{2.4}
Z_{i}(T_{j}|_{Q_{i}\mathcal{H}})Z_{i}^{-1} = T_{j}|_{P_{\widetilde{\Pi}(i)}\mathcal{H}},\quad 1\leq i\leq k, 1\leq j\leq m.
\end{equation}
Setting $Z_{i}:=I|_{P_{i}\mathcal{H}}, k+1\leq i\leq n$ and $Z:=Z_{1}\dot{+}Z_{2}\dot{+}\cdots\dot{+}Z_{n}$,
where $\dot{+}$ denotes the topological direct sum.
From Lemma \ref{2-4}, we know that $$Z\in GL(\mathcal{A}^{\prime}(\mathbf{T}))\quad\text{and}\quad Z|_{P\mathcal{H}}\in GL\left(\mathcal{A}^{\prime}(\mathbf{T}|_{P\mathcal{H}})\right)=GL\left(\bigcap_{i=1}^{m}\mathcal{A}^{\prime}(T_{i}|_{P\mathcal{H}})\right).$$
Therefore, by (\ref{2.4}), we obtain that
$Z|_{P\mathcal{H}}Q_{i}(Z|_{P\mathcal{H}})^{-1} =P_{\widetilde{\Pi}(i)\mathcal{H}}, 1\leq i\leq k$, that is, for any idempotent operator $P$ in $\mathcal{A}^{\prime}(\mathbf{T})$, $\mathbf{T}|_{P\mathcal{H}}$ has unique strongly irreducible decomposition up to similarity.
\end{proof}

Using Lemma \ref{2-8}, we obtain the following result, which is stronger than Lemma \ref{2-3}.

\begin{lemma}\label{2-9}
Let $\mathbf{T}=(T_{1},\ldots,T_{m})\in\mathcal{L}(\mathcal{H})^{m}$ has unique finite strongly irreducible decomposition up to similarity.
Then for any idempotents $P$ and $Q$ in $\mathcal{A}^{\prime}(\mathbf{T})$, $P\sim_{s}(\mathcal{A}^{\prime}(\mathbf{T}))Q$ if and only if $\mathbf{T}|_{P\mathcal{H}}\sim_{s}\mathbf{T}|_{Q\mathcal{H}}$.
\end{lemma}
\begin{proof}
From Lemma \ref{2-3}, we only need to prove the sufficiency of this result.
Since $\mathbf{T}=(T_{1},\ldots,T_{m})\in\mathcal{L}(\mathcal{H})^{m}$ has unique finite strongly irreducible decomposition up to similarity. From Lemma \ref{2-8}, we have that $\mathbf{T}|_{P\mathcal{H}}$, $\mathbf{T}|_{Q\mathcal{H}}$, $\mathbf{T}|_{(I-P)\mathcal{H}}$ and $\mathbf{T}|_{(I-Q)\mathcal{H}}$ also have unique finite strongly irreducible decomposition up to similarity. Since $\mathbf{T}|_{P\mathcal{H}}\sim_{s}\mathbf{T}|_{Q\mathcal{H}}$, there is an operator $X\in GL(\mathcal{L}(P\mathcal{H},Q\mathcal{H}))$ satisfying $X(T_{i}|_{P\mathcal{H}})X^{-1}=T_{i}|_{Q\mathcal{H}}$ for $1\leq i \leq m$.
Setting $\{P_{i}\}_{i=1}^{k}$
is a unit finite strongly irreducible decomposition of $\mathbf{T}|_{P\mathcal{H}}$,
we know that $\{XP_{i}X^{-1}\}_{i=1}^{k}$ is a unit finite strongly irreducible decomposition of $\mathbf{T}|_{Q\mathcal{H}}$ by using Lemma \ref{2-2}.
Let $\{P_{i}\}_{i=k+1}^{n}$ and $\{Q_{i}\}_{i=k+1}^{n}$ be unit finite strongly irreducible decomposition of $\mathbf{T}|_{(I-P)\mathcal{H}}$ and $\mathbf{T}|_{(I-Q)\mathcal{H}}$, respectively. Then $\{P_{i}\}_{i=1}^{n}$ and $\{\{XP_{i}X^{-1}\}_{i=1}^{k},\{Q_{i}\}_{i=k+1}^{n}\}$ are two unit finite strongly irreducible decompositions of $\mathbf{T}$.
From the uniqueness of finite strongly irreducible decomposition of $\mathbf{T}$ up to similarity, we know that there is $Y\in GL(\mathcal{A}^{\prime}(\mathbf{T}))$ such that $\{YP_{i}Y^{-1}\}_{i=1}^{n}$ is a rearrangement of $\{\{XP_{i}X^{-1}\}_{i=1}^{k},\{Q_{i}\}_{i=k+1}^{n}\}$. From Lemma \ref{2-7}, we get that for each $Q_{r}\in\{Q_{i}\}_{i=k+1}^{n}$, there are $P_{r^{\prime}}\in\{P_{i}\}_{i=k+1}^{n}$ and $Z_{r}\in GL(\mathcal{L}(Q_{r}\mathcal{H},P_{r^{\prime}}\mathcal{H}))$ such that
$$Z_{r}(T_{i}|_{Q_{r}\mathcal{H}})Z_{r}^{-1} = T_{i}|_{P_{r^{\prime}}\mathcal{H}},\quad 1\leq i\leq m$$ and
$r_{1}^{\prime}= r_{2}^{\prime}$ if $r_{1}=r_{2}$. Setting $Z_{j}:=X^{-1}|_{(XP_{j}X^{-1})\mathcal{H}},1\leq j\leq k.$
By Lemma \ref{2-4}, we have that $Z=Z_{1}\dot{+}\cdots\dot{+}Z_{n}\in GL(\mathcal{A}^{\prime}(\mathbf{T}))$ and $ZQZ^{-1}=P$, then $P\sim_{s}(\mathcal{A}^{\prime}(\mathbf{T}))Q$.
\end{proof}

\begin{lemma}\label{2-10}
Let $\mathbf{T}=(T_{1},\ldots,T_{m})\in\mathcal{L}(\mathcal{H})^{m}$, and let $P$ and $Q$ be idempotents in $\mathcal{A}^{\prime}(\mathbf{T})$. If $\mathbf{T}|_{P\mathcal{H}}$ is not similar to $\mathbf{T}|_{Q\mathcal{H}}$,
then $P\oplus0_{\mathcal{H}^{(n)}}$ is not similar to $Q\oplus0_{\mathcal{H}^{(n)}}$ in $\mathcal{A}^{\prime}(\mathbf{T}^{(n+1)})=\bigcap\limits_{i=1}^{m}\mathcal{A}^{\prime}(T_{i}^{(n+1)})$ for all $n\in\mathbb{N}$.
\end{lemma}
\begin{proof}
Assume that there is a natural number $n$ such that  $P\oplus0_{\mathcal{H}^{(n)}}$ is similar to $Q\oplus0_{\mathcal{H}^{(n)}}$ in $\mathcal{A}^{\prime}(\mathbf{T}^{(n+1)})$, that is, there is
$X\in GL(\mathcal{A}^{\prime}(\mathbf{T}^{(n+1)}))$ satisfying $X(P\oplus0_{\mathcal{H}^{(n)}})X^{-1}=Q\oplus0_{\mathcal{H}^{(n)}}$.
From Lemma \ref{2-3}, there is an invertible operator $Y\in\mathcal{L}((P\oplus0_{\mathcal{H}^{(n)}})\mathcal{H}^{(n+1)},(Q\oplus0_{\mathcal{H}^{(n)}})\mathcal{H}^{(n+1)})$ such that
$$Y(T_{i}^{(n+1)}|_{(P\oplus0_{\mathcal{H}^{(n)}})\mathcal{H}^{(n+1)}})Y^{-1}=T_{i}^{(n+1)}|_{(Q\oplus0_{\mathcal{H}^{(n)}})\mathcal{H}^{(n+1)}},\quad 1\leq i\leq m.$$
Note that
$$\left(T_{1}^{(n+1)}|_{(P\oplus0_{\mathcal{H}^{(n)}})\mathcal{H}^{(n+1)}}, \ldots ,T_{m}^{(n+1)}|_{(P\oplus0_{\mathcal{H}^{(n)}})\mathcal{H}^{(n+1)}}\right)\sim_{u}\left(T_{1}|_{P\mathcal{H}}, \ldots, T_{m}|_{P\mathcal{H}}\right)$$
and
$$\left(T_{1}^{(n+1)}|_{(Q\oplus0_{\mathcal{H}^{(n)}})\mathcal{H}^{(n+1)}}, \ldots , T_{m}^{(n+1)}|_{(Q\oplus0_{\mathcal{H}^{(n)}})\mathcal{H}^{(n+1)}}\right)
\sim_{u}\left(T_{1}|_{Q\mathcal{H}} ,\ldots , T_{m}|_{Q\mathcal{H}}\right).$$
Thus, $\mathbf{T}|_{P\mathcal{H}}\sim_{s}\mathbf{T}|_{Q\mathcal{H}}$, this is contradictory.
\end{proof}

\begin{lemma}\label{2-11}
Let $\mathbf{T}=(T_{1},\ldots,T_{m})\in\mathcal{L}(\mathcal{H})^{m}$, idempotents $P$ and $Q$ in $\mathcal{A}^{\prime}(\mathbf{T})$ and let $\mathbf{T}^{(n)}$ has unique finite strongly irreducible decomposition up to similarity for any $n\in \mathbb{N}$.
Then $P\sim_{s}(\mathcal{A}^{\prime}(\mathbf{T}))Q$ if and only if $[P]=[Q]$ in $\bigvee(\mathcal{A}^{\prime}(\mathbf{T})).$
\end{lemma}
\begin{proof}
If $P\sim_{s}(\mathcal{A}^{\prime}(\mathbf{T}))Q$, there is an invertible operator $X\in\mathcal{A}^{\prime}(\mathbf{T})$ such that $XPX^{-1}=Q$, and then $[P]=[Q]$ in $\bigvee(\mathcal{A}^{\prime}(\mathbf{T})).$
Conversely, if $[P]=[Q]$ in $\bigvee(\mathcal{A}^{\prime}(\mathbf{T}))$, there is a natural number $n\in \mathbb{N}$ such that
$$P\oplus0_{\mathcal{H}^{(n)}}\sim_{s}(\mathcal{A}^{\prime}(\mathbf{T}^{(n+1)}))Q\oplus0_{\mathcal{H}^{(n)}}.$$
From Lemma \ref{2-3}, we have that
$\mathbf{T}^{(n+1)}|_{(P\oplus0_{\mathcal{H}^{(n)}})\mathcal{H}^{(n+1)}}$ is similar to $\mathbf{T}^{(n+1)}|_{(Q\oplus0_{\mathcal{H}^{(n)}})\mathcal{H}^{(n+1)}}$.
Note that $\mathbf{T}^{(n+1)}|_{(P\oplus0_{\mathcal{H}^{(n)}})\mathcal{H}^{(n+1)}}$ and $\mathbf{T}^{(n+1)}|_{(Q\oplus0_{\mathcal{H}^{(n)}})\mathcal{H}^{(n+1)}}$ are unitary equivalent to $\mathbf{T}|_{P\mathcal{H}}$ and $\mathbf{T}|_{Q\mathcal{H}}$, respectively. Therefore,  $\mathbf{T}|_{P\mathcal{H}}$ is similar to $\mathbf{T}|_{Q\mathcal{H}}$,
and then $P\sim_{s}(\mathcal{A}^{\prime}(\mathbf{T}))Q$ is obtained from Lemma \ref{2-9}.
\end{proof}

\begin{lemma}\label{2-12}
Let $\mathbf{T}=(T_{1},\ldots,T_{m})\in\mathcal{L}(\mathcal{H})^{m}$.
Then $$M_{n}(\mathcal{A}^{\prime}(\mathbf{T}))=\mathcal{A}^{\prime}(\mathbf{T}^{(n)})\quad\text{and} \quad \bigvee(\mathcal{A}^{\prime}(\mathbf{T}^{(n)}))\cong\bigvee(\mathcal{A}^{\prime}(\mathbf{T})),\quad n\in \mathbb{N}.$$
\end{lemma}
\begin{proof}
Note that $M_{n}(\mathcal{A}^{\prime}(\mathbf{T}))\subseteq\mathcal{A}^{\prime}(\mathbf{T}^{(n)})$ is obvious.
For any $((X_{i,j}))_{n\times n}\in\mathcal{A}^{\prime}(\mathbf{T}^{(n)})$, there is
$$((X_{i,j}))_{n\times n}T_{k}^{(n)}=T_{k}^{(n)}((X_{i,j}))_{n\times n},\quad 1\leq k\leq m.$$
So $X_{i,j}T_{k}=T_{k}X_{i,j},1\leq k\leq m.$ This means that $X_{i,j}\in\mathcal{A}^{\prime}(\mathbf{T})$, $((X_{i,j}))_{n\times n}\in M_{n}(\mathcal{A}^{\prime}(\mathbf{T}))$ and
$\mathcal{A}^{\prime}(\mathbf{T}^{(n)})\subseteq M_{n}(\mathcal{A}^{\prime}(\mathbf{T})).$ Thus $M_{n}(\mathcal{A}^{\prime}(\mathbf{T}))=\mathcal{A}^{\prime}(\mathbf{T}^{(n)})$ and $\bigvee(\mathcal{A}^{\prime}(\mathbf{T}^{(n)}))=\bigvee(M_{n}(\mathcal{A}^{\prime}(\mathbf{T})))\cong\bigvee(\mathcal{A}^{\prime}(\mathbf{T})).$
\end{proof}

Let $\mathbf{A}=(A_{1},\ldots,A_{m})\in\mathcal{L}(\mathcal{H})^{m}$ be an $m$-tuple of commuting operators, $\mathcal{H}^{(n)}$ be the direct sum of $n$ copies of $\mathcal{H}$ and $\mathbf{A}^{(n)}=\left(\bigoplus\limits_{1}^{n}A_{1},\ldots,\bigoplus\limits_{1}^{n}A_{m}\right)$ be an $m$-tuple of commuting operators acting on $\mathcal{H}^{(n)}.$
\begin{thm}\label{2-1}
Let $\mathbf{T}=(T_{1},\ldots,T_{m})\in\mathcal{L}(\mathcal{H})^{m}$ be an $m$-tuple of commuting operators. Then the following are equivalent:
\begin{itemize}
  \item [(1)]$\mathbf{T}=(T_{1},\ldots,T_{m})\sim_{s}\left(\bigoplus\limits_{i=1}^{k}A_{1,i}^{(n_{i})},\ldots,\bigoplus\limits_{i=1}^{k}A_{m,i}^{(n_{i})}\right)=
      \bigoplus\limits_{i=1}^{k}\mathbf{A}_{i}^{(n_{i})}=\mathbf{A}$ with respect to the decomposition $\mathcal{H}=\bigoplus\limits_{i=1}^{k}\mathcal{H}_{i}^{(n_{i})}$, where $k,n_{i}<\infty$, $\mathbf{A}_{i}=(A_{1,i},\ldots,A_{m,i})$ is strongly irreducible, $\mathbf{A}_{i}\nsim_{s}\mathbf{A}_{j}$ for $i\neq j$, and for each natural number $n$, $\mathbf{T}^{(n)}$ has unique finite strongly irreducible decomposition up to similarity;
  \item [(2)]$\bigvee(\mathcal{A}^{\prime}(\mathbf{T}))\cong\mathbb{N}^{k}$ and the isomorphism $h:\bigvee(\mathcal{A}^{\prime}(\mathbf{T}))\rightarrow \mathbb{N}^{k}$ sends $[I]$ to $(n_{1},\ldots,n_{k})$. That is, $h([I])=n_{1}e_{1}+\cdots+n_{k}e_{k}$, where $0\neq n_{j}\in\mathbb{N}, 1\leq j\leq k,$ and $\{e_{i}\}_{i=1}^{k}$ are the generators of $\mathbb{N}^{k}$.
\end{itemize}
\end{thm}

\begin{proof}
$(1)\Rightarrow(2)$ Since the Hilbert space $\mathcal{H}=\bigoplus\limits_{i=1}^{k}\mathcal{H}_{i}^{(n_{i})}$, where $k,n_{i}<\infty$. Setting $P_{i}$ is the orthogonal projection onto $\mathcal{H}_{i}$ for $1\leq i\leq k$.
Since $\mathbf{T}^{(n)}$ has unique finite strongly irreducible decomposition up to similarity for any natural number $n\in\mathbb{N}$,
by Lemma \ref{2-8} and Lemma \ref{2-12}, we know that for any idempotent $E\in M_{n}(\mathcal{A}^{\prime}(\mathbf{T}))=\mathcal{A}^{\prime}(\mathbf{T}^{(n)})$, the $m$-tuples of commuting operators $\mathbf{T}^{(n)}|_{E\mathcal{H}}=(T_{1}^{(n)},\ldots,T_{m}^{(n)})|_{E\mathcal{H}}$ and $\mathbf{T}^{(n)}|_{(I-E)\mathcal{H}}=(T_{1}^{(n)},\ldots,T_{m}^{(n)})|_{(I-E)\mathcal{H}}$
are unique finite strongly irreducible decomposition up to similarity. Letting $\{Q_{i}\}_{i=1}^{t}$ and $\{Q_{i}\}_{i=t+1}^{l}$ is a unit finite strongly irreducible decomposition of $\mathbf{T}^{(n)}|_{E\mathcal{H}}$ and $\mathbf{T}^{(n)}|_{(I-E)\mathcal{H}}$, respectively. Then $\{Q_{i}\}_{i=1}^{l}$ is a unit finite strongly irreducible decomposition of $\mathbf{T}^{(n)}$.
Since we have another strongly irreducible decomposition of $\mathbf{T}^{(n)}$ using $nn_{i}$ copies of each of the projections $P_{i}, 1\leq i\leq k$, and $\mathbf{T}^{(n)}$ has unique finite strongly irreducible decomposition up to similarity, we know that there is  an operator $X\in GL(\mathcal{A}^{\prime}(\mathbf{T}^{(n)}))$ such that $XQ_{j}X^{-1}$ is a copy of one of the $P_{i}$, with appropriate multiplicity conditions. In particular, there are integers
$m_{i}$, $0\leq m_{i}\leq nn_{i}$, satisfying
$$XEX^{-1}=X(Q_{1}+Q_{2}+\cdots+Q_{a})X^{-1}=\sum\limits_{i=1}^{k}P_{i}^{(m_{i})}.$$
Define the mapping $h:\bigvee (\mathcal{A}^{\prime}(\mathbf{T}))\longrightarrow\mathbb{N}^{k}$ as
$$[E]\longmapsto(m_{1},m_{2},\ldots,m_{k}).$$
If $[E]=[F]$ in $\bigvee (\mathcal{A}^{\prime}(\mathbf{T}))$, there is an operator $Y\in GL(\mathcal{A}^{\prime}(\mathbf{T}))$ such that $YEY^{-1}=F$. Thus, $F\sim_{s} E\sim_{s} \sum\limits_{i=1}^{k}P_{i}^{(m_{i})}$ and $h([F])=h([E])=(m_{1},\ldots,m_{k})$, this means that the mapping $h$ is well-defined.
Furthermore, if $h([E])= h([F])$, then $E \sim_{s} \sum\limits_{i=1}^{k}P_{i}^{(m_{i})}\sim_{s} F$, which means that the mapping $h$ is one-to-one. For any $k$-tuple of nonnegative integers $(m_{1},\ldots,m_{k})$, there is a natural number $n$, such that for all $i, 1\leq i\leq k$, there are $m_{i}\leq nn_{i}$ and
$h([\sum\limits_{i=1}^{k}P_{i}^{(m_{i})}])=(m_{1},\ldots,m_{k}),$
which shows that the mapping $h$ is onto.
Therefore, we know that $h$ is an isomorphic, and then $\bigvee\left(\mathcal{A}^{\prime}(\mathbf{T})\right)\cong\mathbb{N}^{k}$, and $h([I])=h\left(\left[\sum\limits_{i=1}^{k}P_{i}^{(n_{i})}\right]\right)=(n_{1},\ldots,n_{k})$.

$(2)\Rightarrow(1)$
If $\bigvee(\mathcal{A}^{\prime}(\mathbf{T}))\cong\mathbb{N}^{k}$ and $h$ is the isomorphic mapping from $\bigvee(\mathcal{A}^{\prime}(\mathbf{T}))$ to $\mathbb{N}^{k}$ that satisfies $h([I])=n_{1}e_{1}+\cdots+n_{k}e_{k}$, where $0\neq n_{j}\in\mathbb{N}, 1\leq j\leq k,$ and $\{e_{i}\}_{i=1}^{k}$ are the generators of $\mathbb{N}^{k}$, there is a natural number $r$ and idempotents $\{Q_{i}\}_{i=1}^{k}\subset\mathcal{A}^{\prime}(\mathbf{T}^{(r)})$ satisfy
\begin{equation}\label{14}
h([Q_{i}])=e_{i},\quad 1\leq i\leq k.
\end{equation}
From Lemma \ref{2-12}, we know that $\bigvee(\mathcal{A}^{\prime}(\mathbf{T}^{(n)}))\cong\bigvee(\mathcal{A}^{\prime}(\mathbf{T}))\cong\mathbb{N}^{k}$, so we only need to prove that $\mathbf{T}$ has unique finite strongly irreducible decomposition up to similarity, and other cases can be proved similarly. Let's prove it in four steps:

$\mathbf{Setp~1.}$ For any idempotent $P\in\mathcal{A}^{\prime}(\mathbf{T})$, if $\mathbf{T}|_{P\mathcal{H}}$ is strongly irreducible, there is a $e_{i}\in\{e_{j}\}_{j=1}^{k}$ such that $h([P])=e_{i}$.

Since the mapping $h:\bigvee(\mathcal{A}^{\prime}(\mathbf{T}))\rightarrow \mathbb{N}^{k}$ is isomorphic, by (\ref{14}), there are $\{\lambda_{i}\}_{i=1}^{k}\subset\mathbb{N}$ such that
$$h([P])=\sum\limits_{i=1}^{k}\lambda_{i}e_{i}=\sum\limits_{i=1}^{k}\lambda_{i}h\left([Q_{i}]\right).$$
For $l:=r\sum\limits_{i=1}^{k}\lambda_{i}$, then there is a natural number $n > l$ such that
$$P\oplus0_{\mathcal{H}^{(n-1)}}\sim_{s}(\mathcal{A}^{\prime}(\mathbf{T}^{(n)}))\sum\limits_{i=1}^{k}Q_{i}^{(\lambda_{i})}\oplus0_{\mathcal{H}^{(n-l)}}.$$
From Lemma \ref{2-3}, we have that
$$\mathbf{T}|_{P\mathcal{H}}\sim_{u}\mathbf{T}^{(n)}|_{\left(P\oplus0_{\mathcal{H}^{(n-1)}}\right)\mathcal{H}^{(n)}}\sim_{s}
\mathbf{T}^{(n)}|_{\left(\sum\limits_{i=1}^{k}Q_{i}^{(\lambda_{i})}\oplus0_{\mathcal{H}^{(n-l)}}\right)\mathcal{H}^{(n)}}
\sim_{u}\mathbf{T}|_{\left(\sum\limits_{i=1}^{k}Q_{i}^{(\lambda_{i})}\right)\mathcal{H}^{(l)}}.$$
Thus $\mathbf{T}|_{P\mathcal{H}}\sim_{s}\mathbf{T}|_{\left(\sum\limits_{i=1}^{k}Q_{i}^{(\lambda_{i})}\right)\mathcal{H}^{(l)}}.$
Since $\mathbf{T}|_{P\mathcal{H}}$ is strongly irreducible and the strongly irreducibility remains unchanged under similarity. Therefore, $\mathbf{T}|_{\left(\sum\limits_{i=1}^{k}Q_{i}^{(\lambda_{i})}\right)\mathcal{H}^{(l)}}$ is also strongly irreducible, then there is only one $i\in\{1,2,\ldots,k\}$ such that
$\lambda_{j}=
\left\{
\begin{array}{lc}
1, & j=i,\\
0, & j\neq i.\\
\end{array}
\right.$
It follows that $h([P])=e_{i}$.

$\mathbf{Setp~2.}$ For any idempotent $P$ and $Q$ in $\mathcal{A}^{\prime}(\mathbf{T})$, if $h([P])= h([Q])$, then $\mathbf{T}|_{P\mathcal{H}}\sim_{s}\mathbf{T}|_{Q\mathcal{H}}$.

Setting idempotents $P,Q\in\mathcal{A}^{\prime}(\mathbf{T})$ and
$$h([P])= h([Q])=\sum\limits_{i=1}^{k}\lambda_{i}e_{i}=\sum\limits_{i=1}^{k}\lambda_{i}h([Q_{i}]),$$
 where $\{\lambda_{i}\}_{i=1}^{k}\subset\mathbb{N}$. Letting $w=r\sum\limits_{i=1}^{k}\lambda_{i}$, then there is a natural number $n > w$ such that
$P\oplus0_{\mathcal{H}^{(n-1)}}\sim_{s}(\mathcal{A}^{\prime}(\mathbf{T}^{(n)}))\sum\limits_{i=1}^{k}Q_{i}^{(\lambda_{i})}\oplus0_{\mathcal{H}^{(n-w)}}\sim_{s} Q\oplus0_{\mathcal{H}^{(n-1)}}.$
From Lemma \ref{2-3}, $$\mathbf{T}|_{P\mathcal{H}}\sim_{u}\mathbf{T}^{(n)}|_{\left(P\oplus0_{\mathcal{H}^{(n-1)}}\right)\mathcal{H}^{(n)}}\sim_{s}\mathbf{T}^{(n)}|_{\left(Q\oplus0_{\mathcal{H}^{(n-1)}}\right)\mathcal{H}^{(n)}}
\sim_{u}\mathbf{T}|_{Q\mathcal{H}}.$$
Thus $\mathbf{T}|_{P\mathcal{H}}\sim_{s}\mathbf{T}|_{Q\mathcal{H}}.$

$\mathbf{Setp~3.}$ We prove that $\mathbf{T}$ has finite unit decomposition.

Let $\{P_{i}\}_{i=1}^{s}$ be a unit decomposition of $\mathbf{T}$ and $h([P_{i}])=\sum\limits_{j=1}^{k}\lambda_{i,j}e_{j}$, where $\{\lambda_{i,j}\}_{j=1}^{k}$ are natural numbers. Then there is at least one $\lambda_{i,j}\neq 0$ in $\{\lambda_{i,j}\}_{j=1}^{k}$.  Now we just need to show that $s$ is a finite number. From
$$\sum\limits_{i=1}^{k}n_{i}e_{i}=h([I])=h\left(\left[\sum\limits_{i=1}^{s}P_{i}\right]\right)=\sum\limits_{i=1}^{s}h([P_{i}])=\sum\limits_{i=1}^{s}\sum\limits_{j=1}^{k}\lambda_{i,j}e_{j},$$
we know that $s\leq \sum\limits_{i=1}^{s}\sum\limits_{j=1}^{k}\lambda_{i,j}=\sum\limits_{i=1}^{k}n_{i}<\infty.$ So $\mathbf{T}$ has finite unit decomposition, and then $\mathbf{T}$ has finite strongly irreducible decomposition.

$\mathbf{Setp~4.}$ We prove that $\mathbf{T}$ has unique finite strongly irreducible decomposition up to similarity.

Let $\{P_{i}\}_{i=1}^{t}$ be a unit finite strongly irreducible decomposition of $\mathbf{T}$, then $h\left(\sum\limits_{i=1}^{t}[P_{i}]\right)=h([I])=\sum\limits_{i=1}^{k}n_{i}e_{i}$, with $(i)$, for each $P_{i}\in\{P_{i}\}_{i=1}^{t}$, there is $e_{i^{\prime}}\in\{e_{i}\}_{i=1}^{k}$ such that $h([P_{i}])=e_{i^{\prime}}$. Then we can get $t =\sum\limits_{i=1}^{k}n_{i}<\infty$, and there are $P_{i_{1}},P_{i_{2}},\ldots,P_{i_{n_{i}}}$ in $\{P_{i}\}_{i=1}^{t}$ for each $i,1\leq i\leq k$, which satisfies $h([P_{i_{1}}])=h([P_{i_{2}}])=\cdots=h([P_{i_{n_{i}}}])=e_{i}$. By $(ii),$ we have that
$$\mathbf{T}|_{P_{i_{j}}\mathcal{H}}\sim_{s}\mathbf{T}|_{P_{i_{l}}\mathcal{H}},\quad 1\leq j,l\leq n_{i}.$$
Letting $\mathbf{A}_{i}=\mathbf{T}|_{P_{i_{1}}\mathcal{H}}=(T_{1}|_{P_{i_{1}}\mathcal{H}},T_{2}|_{P_{i_{1}}\mathcal{H}},\ldots,T_{m}|_{P_{i_{1}}\mathcal{H}})$ for $i=1,2,\ldots,k.$ Then
$$\mathbf{T}\sim_{s}\sum\limits_{i=1}^{k}\mathbf{A}_{i}^{(n_{i})},$$
where $k,n_{i}<\infty$ and $\mathbf{A}_{i}\nsim_{s}\mathbf{A}_{j}$ for $i\neq j.$
Let $\{P_{i}^{\prime}\}_{i=1}^{r}$ be another unit finite strongly irreducible decomposition of $\mathbf{T}$.  Similarly, we can also get $r=\sum\limits_{i=1}^{k}n_{i}=t$, and for each $i,1\leq i\leq k,$ there are $P_{i_{1}}^{\prime},P_{i_{2}}^{\prime},\ldots,P_{i_{n_{i}}}^{\prime}$ in $\{P_{i}^{\prime}\}_{i=1}^{t}$ such that $h([P_{i_{1}}])=h([P_{i_{2}}])=\cdots=h([P_{i_{n_{i}}}])=e_{i}$ and $\mathbf{T}|_{P_{i_{j}}\mathcal{H}}\sim_{s}\mathbf{T}|_{P_{i_{l}}\mathcal{H}}$ for $1\leq j,l\leq n_{i}.$
Therefore, there is a permutation $\Pi\in S_{t}$ such that for $i=1,2,\ldots,t$, $h([P_{i}])=h([P_{\Pi(i)}^{\prime}])$. By $(ii)$,
$\mathbf{T}|_{P_{i}\mathcal{H}}\sim_{s}\mathbf{T}|_{P_{\Pi(i)}^{\prime}\mathcal{H}}$, then there is $X_{i}\in GL(\mathcal{L}(P_{i}\mathcal{H},P_{\Pi(i)}^{\prime}\mathcal{H}))$ such that $X_{i}(\mathbf{T}|_{P_{i}\mathcal{H}})X_{i}^{-1} =\mathbf{T}|_{P_{\Pi(i)}^{\prime}\mathcal{H}}$. From Lemma \ref{2-4}, $X=X_{1}\dot{+}X_{2}\dot{+}\cdots\dot{+}X_{n}\in GL(\mathcal{A}^{\prime}(\mathbf{T}))$ and $XP_{i}X^{-1} = P_{\Pi(i)}^{\prime}$ for $i=1,2,\ldots,t.$ That is, $\mathbf{T}$ has unique finite strongly irreducible decomposition up to similarity.
\end{proof}

\begin{corollary}\label{2-13}
Let $\widehat{\mathbf{T}}=(T_{1},T_{2},\ldots,T_{m}), \mathbf{\widetilde{T}}=(\widetilde{T}_{1},\widetilde{T}_{2},\ldots,\widetilde{T}_{m})\in\mathcal{L}(\mathcal{H})^{m}$ be strongly irreducible  m-tuples, and let
$\mathbf{T}=\widehat{\mathbf{T}}\oplus\mathbf{\widetilde{T}}=(\left(\begin{pmatrix}T_{1}&0\\ 0&\widetilde{T}_{1}\end{pmatrix},\begin{pmatrix}T_{2}&0\\ 0&\widetilde{T}_{2}\end{pmatrix}\cdots,\begin{pmatrix}T_{m}&0\\ 0&\widetilde{T}_{m}\end{pmatrix}\right)$. Then the following properties hold:
\begin{itemize}
  \item [(1)]$\widehat{\mathbf{T}}\sim_{s}\mathbf{\widetilde{T}}$ if and only if $\bigvee(\mathcal{A}^{\prime}(\mathbf{T}))\cong\mathbb{N}$.
  \item [(2)]If for any natural number $n$, $\mathbf{T}^{(n)}$ has unique finite strongly irreducible decomposition up to similarity. Then $\widehat{\mathbf{T}}\sim_{s}\mathbf{\widetilde{T}}$ if and only if $K_{0}(\mathcal{A}^{\prime}(\mathbf{T}))\cong\mathbb{Z}.$
\end{itemize}
\end{corollary}
\begin{proof}
Using Theorem \ref{2-1}, we can directly obtain that $(1)$ holds.
Since $K_{0}(\mathbb{N})=\mathbb{Z}$, there is $K_{0}(\mathcal{A}^{\prime}(\mathbf{T}))\cong\mathbb{Z}$ when $\widehat{\mathbf{T}}$ is similar to $\mathbf{\widetilde{T}}$.
On the contrary,
 since $\widehat{\mathbf{T}}$ and $\mathbf{\widetilde{T}}$ are strongly irreducible operator tuples, we have that
$\bigvee(\mathcal{A}^{\prime}(\mathbf{T}))\cong\mathbb{N}^{k}$, $0<k\leq 2$. And since $K_{0}(\mathcal{A}^{\prime}(\mathbf{T}))\cong\mathbb{Z}$, $k=1$. So $\widehat{\mathbf{T}}\sim_{s}\mathbf{\widetilde{T}}$.
\end{proof}

\section{The similarity of Cowen-Douglas operator tuples with index one }
\begin{lemma}[Ameer Athavale]\label{lm2}
Any two spherical shifts are unitarily equivalent in the sense that a single unitary operator intertwines their corresponding operator coordinates.
\end{lemma}
The $m$-shift $\mathbf{M}_{z}=(M_{z_{1}},\ldots,M_{z_{m}})$ of multiplication by coordinate functions $z_{i}$ on the Drury-Arveson space $H_{m}^{2}$, to be referred to as the \textit{Szeg\"{o} tuple}, is a classical model of the spherical shift.

\begin{thm}\label{0.7}
Let $\mathbf{T}=(T_{1},\ldots,T_{m})\in\mathcal{L}(\mathcal{H})^{m}$ be a m-tuple which satisfies the following conditions:
\begin{itemize}
  \item [(1)] $\sum\limits_{i=1}^{m}T_{i}^{*}T_{i}$ is a projection, and
  \item [(2)] if $x_{1},\ldots,x_{m}\in\mathcal{H}$ with $T_{i}x_{j}=T_{j}x_{i}$ for all $i,j$, then there is an $x\in\mathcal{H}$ with $x_{i}=T_{i}x$ for all $i$.
\end{itemize}
Then $T$ is unitarily equivalent to $\mathbf{S}^{*}\oplus \mathbf{V}$, where $\mathbf{S}^{*}$ is a direct sum of the backwards multishifts on Drury-Arveson space, and $\mathbf{V}$ is a spherical isometry.
\end{thm}

For the $m$-tuple $\mathbf{T}=(T_{1},\ldots,T_{m})\in\mathcal{L}(\mathcal{H})^{m}$ in the above theorem.
Let $E_{0}=\bigcap\limits_{i=1}^{m}\ker T_{i}$. Inductively define a sequence of positive operators by
$$P_{0}=I \quad \text{and}\quad P_{n+1}=\sum\limits_{i=1}^{m}T_{i}^{*}P_{n}T_{i},\quad n\geq0.$$
It follows that for $n\geq 1,$
$$P_{n}=\sum\limits_{|\alpha|=n}\begin{pmatrix} n \\ \alpha \end{pmatrix}\mathbf{T}^{*\alpha}\mathbf{T}^{\alpha}\quad\text{and}\quad\ker P_{n}=\bigcap\limits_{|\alpha|=n}\ker \mathbf{T}^{\alpha}.$$
Note that for $n\geq1$ we have $P_{n}-P_{n+1}=\sum\limits_{i=1}^{m}T_{i}^{*}(P_{n-1}-P_{n})T_{i}$. Therefore, from part $(1)$ of Theorem \ref{0.7} and the induced argument, it means that the sequence $\{P_{n}\}_{n\in\mathbb{N}}$ is a nonincreasing sequence of positive operators and converges strongly to a positive operator $P$. From the proof process of Theorem \ref{0.7} (see [1]), we can get that $\mathcal{M}=ran P$ reduces every $T_{i}$, $\mathbf{T}|_{\mathcal{M}}$ is a spherical isometry and $\mathbf{T}^{\ast}|_{\mathcal{M}^{\bot}}$ is unitarily equivalent to the $m$-shift acting on $H_{m}^{2}(E_{0}).$

\begin{lemma}\label{lm1}
The $m$-tuple $\mathbf{M}_{z}^{*}=(M_{z_{1}}^{*},\ldots,M_{z_{m}}^{*})$ on the Drury-Arveson space $H_{m}^{2}$ satisfies the condition of Theorem \ref{0.7} and the sequence $\left\{P_{n}=\sum\limits_{|\alpha|=n}\begin{pmatrix} n \\ \alpha \end{pmatrix}\mathbf{M}_{z}^{\alpha}(\mathbf{M}_{z}^{*})^{\alpha}\right\}_{n=1}^{\infty}$ strongly converges to zero.
\end{lemma}
\begin{proof}
First, we can know that $\sum\limits_{i=1}^{m}M_{z_{i}}M^{*}_{z_{i}}$ is a projection operator from $\sum\limits_{i=1}^{m}M_{z_{i}}M^{*}_{z_{i}}=I-1\otimes1.$
Second, we  need to prove that if $x_{1},\ldots,x_{m}\in H_{m}^{2}$ with $M^{*}_{z_{i}}x_{j}=M^{*}_{z_{j}}x_{i}$ for all $i,j$, then there is an $x\in H_{m}^{2}$ with $x_{i}=M^{*}_{z_{i}}x$ for all $i$.
Because $\{z^{\alpha}\}_{\alpha\in\mathbb{N}^{m}}$ is a set of bases of the space $H_{m}^{2}$, for any $y\in (ran M^{*}_{z_{i}})^{\bot}$ and $\alpha\in\mathbb{N}^{m}$, we have
$$0=\langle y, M^{*}_{z_{i}}z^{\alpha}\rangle=\langle M_{z_{i}}y, z^{\alpha}\rangle.$$
Therefore, $y\in\ker M_{z_{i}}$. As we all know, $M_{z_{i}}$ is injective, so $y=0$,  which means that $M^{*}_{z_{i}}$ is surjective.
Without losing generality, assume that
$$x_{i}=\sum\limits_{\alpha\in\mathbb{N}^{m}}f_{i}(\alpha)z^{\alpha},\quad 1\leq i\leq m,$$
and we get
$$\sum\limits_{\alpha\in\mathbb{N}^{m}}\frac{\alpha_{i}+1}{|\alpha|+1}f_{j}(\alpha+e_{i})z^{\alpha}=M^{*}_{z_{i}}x_{j}=M^{*}_{z_{j}}x_{i}
=\sum\limits_{\alpha\in\mathbb{N}^{m}}\frac{\alpha_{j}+1}{|\alpha|+1}f_{i}(\alpha+e_{j})z^{\alpha}.$$
It follows that
\begin{equation}\label{9}
f_{j}(\alpha+e_{i})=\frac{\alpha_{j}+1}{\alpha_{i}+1}f_{i}(\alpha+e_{j})
\end{equation}
for all $\alpha\in\mathbb{N}^{m}$ and $1\leq i,j\leq m$.
From $M^{*}_{z_{1}}$ is surjective, there exists $y_{1}=\sum\limits_{\alpha\in\mathbb{N}^{m}}g_{1}(\alpha)z^{\alpha}\in H_{m}^{2}$ such that $x_{1} = M^{*}_{z_{1}}y_{1}$, that is, $$\sum\limits_{\alpha\in\mathbb{N}^{m}}f_{1}(\alpha)z^{\alpha}=M^{*}_{z_{1}}\sum\limits_{\alpha\in\mathbb{N}^{m}}g_{1}(\alpha)z^{\alpha}
=\sum\limits_{\alpha\in\mathbb{N}^{m}\atop\alpha_{1}>0}g_{1}(\alpha)\frac{\alpha_{1}}{|\alpha|}z^{\alpha-e_{1}}=\sum\limits_{\alpha\in\mathbb{N}^{m}}g_{1}(\alpha+e_{1})\frac{\alpha_{1}+1}{|\alpha|+1}z^{\alpha}.$$
By comparing the coefficients on both sides of the above formula,
\begin{equation}\label{10}
g_{1}(\alpha+e_{1})=\frac{|\alpha|+1}{\alpha_{1}+1}f_{1}(\alpha),\quad \alpha\in\mathbb{N}^{m}.
\end{equation}
Letting $$h_{1}:=\sum\limits_{\alpha\in\mathbb{N}^{m}\atop \alpha_{1}>0}g_{1}(\alpha)z^{\alpha},$$ then $h_{1}\in H_{m}^{2}$ and $x_{1} = M^{*}_{z_{1}}h_{1}$.
By (\ref{9}) and (\ref{10}), we know that for $ 2\leq i\leq m,$
$$\begin{array}{lll}
M^{*}_{z_{i}}h_{1}&=&\sum\limits_{\alpha\in\mathbb{N}^{m}\atop \alpha_{1}>0,\alpha_{i}>0}g_{1}(\alpha)\frac{\alpha_{i}}{|\alpha|}z^{\alpha-e_{i}}\\[5pt]
&=&\sum\limits_{\alpha\in\mathbb{N}^{m}}g_{1}(\alpha+e_{i}+e_{1})\frac{\alpha_{i}+1}{|\alpha|+2}z^{\alpha+e_{1}}\\[5pt]
&=&\sum\limits_{\alpha\in\mathbb{N}^{m}}\frac{\alpha_{i}+1}{\alpha_{1}+1}f_{1}(\alpha+e_{i})z^{\alpha+e_{1}}\\[5pt]
&=&\sum\limits_{\alpha\in\mathbb{N}^{m}\atop \alpha_{1}>0}f_{i}(\alpha)z^{\alpha}.
\end{array}$$
It follows that
$$x_{i}=\sum\limits_{\alpha\in\mathbb{N}^{m}\atop\alpha_{1}=0}f_{i}(\alpha)z^{\alpha}+M^{*}_{z_{i}}h_{1},\quad 2\leq i\leq m.$$
Set $$\gamma_{2}:=\sum\limits_{\alpha\in\mathbb{N}^{m}\atop\alpha_{1}=0}f_{2}(\alpha)z^{\alpha}\in H_{m}^{2}.$$
Since $M^{*}_{z_{2}}$ is surjective, there exists $y_{2}=\sum\limits_{\alpha\in\mathbb{N}^{m}}g_{2}(\alpha)z^{\alpha}\in H_{m}^{2}$ such that $\gamma_{2} = M^{*}_{z_{2}}y_{2}$, that is,
$$\sum\limits_{\alpha\in\mathbb{N}^{m}\atop \alpha_{1}=0}f_{2}(\alpha)z^{\alpha}=M^{*}_{z_{2}}\sum\limits_{\alpha\in\mathbb{N}^{m}}g_{2}(\alpha)z^{\alpha}
=\sum\limits_{\alpha\in\mathbb{N}^{m}\atop\alpha_{2}>0}g_{2}(\alpha)\frac{\alpha_{2}}{|\alpha|}z^{\alpha-e_{2}}=\sum\limits_{\alpha\in\mathbb{N}^{m}}g_{2}(\alpha+e_{2})\frac{\alpha_{2}+1}{|\alpha|+1}z^{\alpha},$$
and then
\begin{equation}\label{11}
g_{2}(\alpha+e_{2})=\frac{|\alpha|+1}{\alpha_{2}+1}f_{2}(\alpha)
\end{equation}
for all $\alpha=(\alpha_{1},\alpha_{2},\ldots,\alpha_{m})\in\mathbb{N}^{m}$ with $\alpha_{1}=0.$
Letting
$$h_{2}:=\sum\limits_{\alpha\in\mathbb{N}^{m}\atop \alpha_{1}=0,\alpha_{2}>0}g_{2}(\alpha)z^{\alpha},$$
then $h_{2}\in H_{m}^{2},$ $\gamma_{2} = M^{*}_{z_{2}}h_{2}$ and $M^{*}_{z_{1}}h_{2}=0$. Thus,
$$x_{i}=M^{*}_{z_{i}}(h_{1}+h_{2}),\quad i=1,2.$$
From (\ref{9}) and (\ref{11}), for $3\leq i\leq m,$
$$\begin{array}{lll}
M^{*}_{z_{i}}h_{2}&=&\sum\limits_{\alpha\in\mathbb{N}^{m}\atop \alpha_{2},\alpha_{i}>0,\alpha_{1}=0}g_{2}(\alpha)\frac{\alpha_{i}}{|\alpha|}z^{\alpha-e_{i}}\\[5pt]
&=&\sum\limits_{\alpha\in\mathbb{N}^{m}\atop\alpha_{1}=0}g_{2}(\alpha+e_{i}+e_{2})\frac{\alpha_{i}+1}{|\alpha|+2}z^{\alpha+e_{2}}\\
&=&\sum\limits_{\alpha\in\mathbb{N}^{m}\atop\alpha_{1}=0}\frac{\alpha_{i}+1}{\alpha_{2}+1}f_{2}(\alpha+e_{i})z^{\alpha+e_{2}}\\
&=&\sum\limits_{\alpha\in\mathbb{N}^{m}\atop \alpha_{1}=0,\alpha_{2}>0}f_{i}(\alpha)z^{\alpha}.
\end{array}$$
So
$$x_{i}=\sum\limits_{\alpha\in\mathbb{N}^{m}\atop\alpha_{1}=\alpha_{2}=0}f_{i}(\alpha)z^{\alpha}+M^{*}_{z_{i}}(h_{1}+h_{2}),\quad 3\leq i\leq m.$$
From the above description, using the inductive method, we can assume that for integer $k < m$, there are $h_{j}\in H_{m}^{2}, (1\leq j\leq k)$ such that $$x_{i}=M^{*}_{z_{i}}(\sum\limits_{j=1}^{k}h_{j}),\quad 1\leq i\leq k$$
and
$$x_{i}=\sum\limits_{\alpha\in\mathbb{N}^{m}\atop\alpha_{1}=\cdots=\alpha_{k}=0}f_{i}(\alpha)z^{\alpha}+M^{*}_{z_{i}}(\sum\limits_{j=1}^{k}h_{j}),\quad k+1\leq i\leq m.$$
For
$x_{k+1}=\sum\limits_{\alpha\in\mathbb{N}^{m}\atop\alpha_{1}=\cdots=\alpha_{k}=0}f_{k+1}(\alpha)z^{\alpha}+M^{*}_{z_{k+1}}(\sum\limits_{j=1}^{k}h_{j}),$
let $$\gamma_{k+1}:=\sum\limits_{\alpha\in\mathbb{N}^{m}\atop\alpha_{1}=\cdots=\alpha_{k}=0}f_{k+1}(\alpha)z^{\alpha}.$$
Because $\gamma_{k+1}\in H_{m}^{2}$ and
$M^{*}_{z_{k+1}}$ is surjective, there exists $y_{k+1}=\sum\limits_{\alpha\in\mathbb{N}^{m}}g_{k+1}(\alpha)z^{\alpha}\in H_{m}^{2}$ such that $\gamma_{k+1} = M^{*}_{z_{k+1}}y_{k+1}$, that is,
$$\sum\limits_{\alpha\in\mathbb{N}^{m}\atop \alpha_{1}=\cdots=\alpha_{k}=0}f_{k+1}(\alpha)z^{\alpha}=\sum\limits_{\alpha\in\mathbb{N}^{m}\atop\alpha_{k+1}>0}g_{k+1}(\alpha)\frac{\alpha_{k+1}}{|\alpha|}z^{\alpha-e_{k+1}}=
\sum\limits_{\alpha\in\mathbb{N}^{m}}g_{k+1}(\alpha+e_{k+1})\frac{\alpha_{k+1}+1}{|\alpha|+1}z^{\alpha},$$
and then
\begin{equation}\label{12}
g_{k+1}(\alpha+e_{k+1})=\frac{|\alpha|+1}{\alpha_{k+1}+1}f_{k+1}(\alpha)
\end{equation}
for all $\alpha=(\alpha_{1},\alpha_{2},\ldots,\alpha_{m})\in\mathbb{N}^{m}$ with $\alpha_{1}=\alpha_{2}=\cdots=\alpha_{k}=0.$
Letting
$$h_{k+1}:=\sum\limits_{\alpha\in\mathbb{N}^{m}\atop \alpha_{1}=\cdots=\alpha_{k}=0,\alpha_{k+1}>0}g_{k+1}(\alpha)z^{\alpha},$$
then $h_{k+1}\in H_{m}^{2}$ and $\gamma_{k+1} = M^{*}_{z_{k+1}}h_{k+1}$ and $M^{*}_{z_{i}}h_{k+1}=0$ for $1\leq i\leq k$. Thus,
$$x_{i}=M^{*}_{z_{i}}(\sum\limits_{j=1}^{k+1}h_{j}),\quad 1\leq i\leq k+1$$
If $k+1< m,$ from (\ref{9}) and (\ref{12}), we know that for $k+2\leq i\leq m,$
$$\begin{array}{lll}
M^{*}_{z_{i}}h_{k+1}&=&\sum\limits_{\alpha\in\mathbb{N}^{m}\atop \alpha_{1}=\cdots=\alpha_{k}=0,\alpha_{k+1}>0,\alpha_{i}>0}g_{k+1}(\alpha)\frac{\alpha_{i}}{|\alpha|}z^{\alpha-e_{i}}\\[5pt]
&=&\sum\limits_{\alpha\in\mathbb{N}^{m}\atop\alpha_{1}=\cdots=\alpha_{k}=0}g_{k+1}(\alpha+e_{i}+e_{k+1})\frac{\alpha_{i}+1}{|\alpha|+2}z^{\alpha+e_{k+1}}\\
&=&\sum\limits_{\alpha\in\mathbb{N}^{m}\atop\alpha_{1}=\cdots=\alpha_{k}=0}\frac{\alpha_{i}+1}{\alpha_{k+1}+1}f_{k+1}(\alpha+e_{i})z^{\alpha+e_{k+1}}\\
&=&\sum\limits_{\alpha\in\mathbb{N}^{m}\atop \alpha_{1}=\cdots=\alpha_{k}=0,\alpha_{k+1}>0}f_{i}(\alpha)z^{\alpha}.
\end{array}$$
Therefore,  for $k+2\leq i\leq m,$
$$x_{i}=\sum\limits_{\alpha\in\mathbb{N}^{m}\atop\alpha_{1}=\cdots=\alpha_{k+1}=0}f_{i}(\alpha)z^{\alpha}+M^{*}_{z_{i}}(\sum\limits_{j=1}^{k+1}h_{j}).$$
In this way, we just need to let $x:=\sum\limits_{j=1}^{m}h_{j}$, then $x\in H_{m}^{2}$ and $x_{i}=M^{*}_{z_{i}}x$ for all $i=1,2,\ldots,m.$

Finally, we prove that the sequence of positive operators $\left\{P_{n}=\sum\limits_{|\alpha|=n}\begin{pmatrix} n \\ \alpha \end{pmatrix}\mathbf{M}^{\alpha}_{z}\mathbf{M}^{*\alpha}_{z}\right\}_{n=0}^{\infty}$ strongly converges to zero. Since $\ker P_{n}=\bigcap\limits_{|\alpha|=n}\ker \mathbf{M}^{*\alpha}_{z}$, we just have to prove that
$$\lim\limits_{n}\ker P_{n}=H_{m}^{2}=span\{z^{\alpha} | \alpha\in\mathbb{N}^{m}\}.$$
For any $z^{\alpha}$, we only need to take $|\beta|=n>|\alpha|$ to get $\mathbf{M}^{*\beta}_{z}z^{\alpha}=0$. Hence, for  $n>|\alpha|,$
$$P_{n}z^{\alpha}=\sum\limits_{|\beta|=n}\begin{pmatrix} n \\ \beta \end{pmatrix}\mathbf{M}^{\beta}_{z}\mathbf{M}^{*\beta}_{z}z^{\alpha}=0.$$
So that means that the sequence of positive operators $\{P_{n}\}_{n=0}^{\infty}$ strongly converges to zero.
\end{proof}

\begin{corollary}\label{cor1}
Let $n$ be a positive integer, $\mathcal{H}=H^{2}_{m}$ and $P\in\mathcal{A}^{\prime}(\mathbf{M}_{z}^{*(n)})=\bigcap\limits_{i=1}^{m}\mathcal{A}^{\prime}(M_{z_{i}}^{*(n)})$ be an idempotent, then
both $\mathbf{M}_{z}^{*(n)}$ and $\mathbf{M}_{z}^{*(n)}|_{P\mathcal{H}^{(n)}}$ satisfy the condition of Theorem \ref{0.7}, and the corresponding sequence of positive operators $\{P_{n}\}_{n=0}^{\infty}$ also strongly converge to zero.
\end{corollary}
\begin{proof}
Using Lemma \ref{lm1}, it is obvious that $\mathbf{M}_{z}^{*(n)}$ satisfies this conclusion.
For any $x\in P\mathcal{H}^{(n)}\subseteq\mathcal{H}^{(n)}$ and $i=1,\ldots,m$,
$$M_{z_{i}}^{*(n)}x=M_{z_{i}}^{*(n)}|_{P\mathcal{H}^{(n)}}x=M_{z_{i}}^{*(n)}Px.$$
Letting $T_{i}^{\ast}:=M_{z_{i}}^{*(n)}|_{P\mathcal{H}^{(n)}},$ from $P\in\mathcal{A}^{\prime}(M_{z}^{*(n)})$ is an idempotent,
$$\sum\limits_{i=1}^{m}T_{i}T^{*}_{i}x=\sum\limits_{i=1}^{m}P^{\ast}M_{z_{i}}^{(n)}M^{*(n)}_{z_{i}}Px=\sum\limits_{i=1}^{m}M_{z_{i}}^{(n)}M^{*(n)}_{z_{i}}x$$
for any $x\in P\mathcal{H}^{(n)}$.
Therefore, we get that $\sum\limits_{i=1}^{m}T_{i}T^{*}_{i}$ is the projection from $\sum\limits_{i=1}^{m}M_{z_{i}}^{(n)}M^{*(n)}_{z_{i}}$ is a projection.
If $y_{1},\ldots,y_{m}\in P\mathcal{H}^{(n)}\subseteq\mathcal{H}^{(n)}$ satisfy $T_{i}^{\ast}y_{j}=T_{j}^{\ast}y_{i}$ for all $i,j,$ then
$$M^{*(n)}_{z_{i}}y_{j}=M^{*(n)}_{z_{i}}|_{P\mathcal{H}^{(n)}}y_{j}=T_{i}^{\ast}y_{j}=T_{j}^{\ast}y_{i}=M^{*(n)}_{z_{j}}|_{P\mathcal{H}^{(n)}}y_{i}=M^{*(n)}_{z_{j}}y_{i}.$$
From $\mathbf{M}_{z}^{*(n)}$ is a pure isometry, there exists $h\in\mathcal{H}^{(n)}$ such that $y_{i}=M^{*(n)}_{z_{i}}h$. Then $$y_{i}=Py_{i}=PM^{*(n)}_{z_{i}}h=M^{*(n)}_{z_{i}}Ph,\quad 1\leq i \leq m.$$
Set $y:=Ph$, then $y\in PH_{m}^{2}$ and
$$y_{i}=T_{i}^{\ast}y,\quad 1\leq i \leq m.$$
For the sequence of positive operators
$$P_{0}=I_{\mathcal{H}^{(n)}} \quad \text{and}\quad P_{n+1}=\sum\limits_{i=1}^{m}T_{i}P_{n}T_{i}^{*},\quad n\geq0.$$
We know that $\ker P_{n}=\bigcap\limits_{|\alpha|=n}\ker \mathbf{T}^{*\alpha}$, where $\mathbf{T}^{\ast}=(T^{\ast}_{1},\ldots,T^{\ast}_{m})$. Now we just need to prove that
$$\lim\limits_{n}\ker P_{n}=P\mathcal{H}^{(n)}=span\{z^{\alpha} | \alpha\in\mathbb{N}^{m}, z^{\alpha}\in P\mathcal{H}^{(n)}\}.$$
For any $z^{\alpha}\in P\mathcal{H}^{(n)}$, we only need to take $|\beta|=n>|\alpha|$ to get $(T^{\ast})^{\beta}z^{\alpha}=0$. Hence, for $n>|\alpha| ,$
$$P_{n}z^{\alpha}=\sum\limits_{|\beta|=n}\begin{pmatrix} n \\ \beta \end{pmatrix}\mathbf{T}^{\beta}(\mathbf{T}^{\beta})^{*}z^{\alpha}=0,\quad z^{\alpha}\in P\mathcal{H}^{(n)}.$$
So that means that the sequence of positive operators $\{P_{n}\}_{n=0}^{\infty}$ strongly converges to zero.
\end{proof}

\begin{lemma}\label{0.6}
Let $\mathbf{T}=(T_{1},\ldots,T_{m})\in\mathcal{L}(\mathcal{H})^{m}$ be a m-tuple which satisfies the following conditions:
\begin{itemize}
  \item [(1)] $\sum\limits_{i=1}^{m}T_{i}^{*}T_{i}$ is a projection,
  \item [(2)] if $x_{1},\ldots,x_{m}\in\mathcal{H}$ with $T_{i}x_{j}=T_{j}x_{i}$ for all $i,j$, then there is an $x\in\mathcal{H}$ with $x_{i}=T_{i}x$ for all $i$, and
  \item [(3)] the sequence $\{P_{n}=\sum\limits_{|\alpha|=n}\begin{pmatrix} n \\ \alpha \end{pmatrix}\mathbf{T}^{*\alpha}\mathbf{T}^{\alpha}\}_{n=1}^{\infty}$ converges strongly to zero.
\end{itemize}
Then $\mathbf{T}$ is unitary equivalent to $\mathbf{M}_{z}^{*(k)},$
where $k=\dim \ker \mathbf{T}=\dim (\bigcap\limits_{i=1}^{m}\ker T_{i})$.
\end{lemma}

\begin{lemma}\label{6}
Let $\mathbf{T}=(T_{1},\ldots,T_{m})\in\mathbf{\mathcal{B}}_{n}(\Omega)$ satisfies the condition of Lemma \ref{0.6}, then
\begin{itemize}
  \item [(1)] $\mathbf{T}\sim_{u}\mathbf{M}_{z}^{*(l)}$ and $\mathbf{T}\in\mathbf{\mathcal{B}}_{l}(\mathbb{B}^{m})$, where $l=\dim \ker \mathbf{T}=\dim (\bigcap\limits_{i=1}^{m}\ker T_{i})$
  \item [(2)] $\mathbf{T}$ is a strongly irreducible if and only if $\mathbf{T}\in\mathbf{\mathcal{B}}_{1}(\mathbb{B}^{m})$.
\end{itemize}
\end{lemma}

\begin{proof}
By Lemma \ref{0.6}, we have that $\mathbf{T}\sim_{u}\mathbf{M}_{z}^{*(l)}$ where $l=\dim \ker \mathbf{T}=\dim (\bigcap\limits_{i=1}^{m}\ker T_{i})$. Since $\mathbf{M}_{z}^{*(l)}\in\mathbf{\mathcal{B}}_{l}(\mathbb{B}^{m})$ and $\mathbf{T}\sim_{u}\mathbf{M}_{z}^{*(l)}$, $\mathbf{T}\in\mathbf{\mathcal{B}}_{l}(\mathbb{B}^{m})$. At the same time, we obtain that $\mathbf{T}$ is strongly irreducible if and only if $\mathbf{T}\sim_{u}\mathbf{M}_{z}^{*}$, that is, $\mathbf{T}\in\mathbf{\mathcal{B}}_{1}(\mathbb{B}^{m})$.
\end{proof}

\begin{lemma}\label{0.8}
Let $P\in\mathcal{A}^{\prime}(\mathbf{M}_{z}^{*(n)})=\bigcap\limits_{i=1}^{m}\mathcal{A}^{\prime}(M_{z_{i}}^{*(n)})$ be an idempotent, $\mathcal{H}=H^{2}_{m}$, $\mathbf{T}=\mathbf{M}_{z}^{*(n)}|_{P\mathcal{H}^{(n)}}$ and $l=\dim \ker \mathbf{T}=\dim \left(\bigcap\limits_{i=1}^{m}\ker T_{i}\right)$. Then there is a unitary operator $U$ such that
\begin{itemize}
  \item [(1)] $U(P\mathcal{H}^{(n)})=\mathcal{H}^{(l)}\oplus0^{(n-l)}$, that is,
  $$UPU^{*}=\begin{pmatrix} I_{\mathcal{H}^{(l)}}\,\, & \star\\ 0 \,\,& 0\end{pmatrix}\begin{matrix} \mathcal{H}^{(l)}\\ \quad \mathcal{H}^{(n-l)} \end{matrix}.$$
  \item [(2)] Let $V=U|_{P\mathcal{H}^{(n)}}$, then $V\mathbf{T}V^{*}=\mathbf{M}_{z}^{*(l)}$, in other words, $\mathbf{T}$ is unitarily equivalent to $\mathbf{M}_{z}^{*(l)}$.
\end{itemize}
\end{lemma}

\begin{proof}\label{0.9}
By Lemma \ref{cor1} and Lemma \ref{6}, $T$ is unitarily equivalent to $\mathbf{M}_{z}^{*(l)}$, where $l=\dim \ker \mathbf{T}$. Thus there is a unitary operator
$$V : P\mathcal{H}^{(n)}\longrightarrow\mathcal{H}^{(l)}$$
such that
$$V\mathbf{T}V^{*}=M_{z}^{*(l)}.$$
Note that if $l<n$, $\mathcal{H}^{(n)}\ominus P\mathcal{H}^{(n)}$ is infinite dimensional. Therefore, there exists a unitary operator
$$W : \mathcal{H}^{(n)}\ominus P\mathcal{H}^{(n)}\longrightarrow\mathcal{H}^{(n-l)}.$$
Set $U:=V\oplus W$, then $U$  satisfies the lemma.
\end{proof}

For $\mathbf{T}=(T_{1},\ldots,T_{m})\in\mathbf{\mathcal{B}}_{n}(\Omega)$ and $z=(z_{1},\ldots,z_{m})\in\Omega$. Let $A\in\mathcal{A}^{\prime}(\mathbf{T})$, then
$$A(\mathbf{T}-zI)=(\mathbf{T}-zI)A$$ and
$$A\ker(\mathbf{T}-zI)\subset\ker(\mathbf{T}-zI).$$
Define
$$(\Gamma_{\mathbf{T}}A)(z):=A|_{\ker(\mathbf{T}-zI)}$$
where $A\in\mathcal{A}^{\prime}(\mathbf{T})$. Then $\Gamma_{\mathbf{T}}$ is an injective contraction.

We know that studying the commutant of operators is helpful to understand the structure of operators, In \cite{SW}, Shields and Wallen proved that the commutant of a contractive multiplication operator on a reproducing kernel Hilbert space of scalar-valued holomorphic functions on the open unit disc is the algebra of bounded holomorphic functions on the open unit disc. Chavan, Podder and Trivedi in \cite{CPT2018} pointed out that under some conditions, the commutant of the multiplication $m$-tuple $\mathbf{M}_{z}$ on a reproducing kernel Hilbert space of $E$-valued holomorphic functions on a bounded domain $\Omega$ is the algebra $H^{\infty}_{B(E)}(\Omega)$ of bounded holomorphic $B(E)$-valued functions on $\Omega$, where $E$ is a separable Hilbert space.
Inspired by these results, we give the commutant of adjoint $\mathbf{M}_{z}^{*}=(M_{z_{1}}^{*},\ldots,M_{z_{m}}^{*})$ of $m$-tuple of multiplication operators on Hilbert space $A_{k}^{2},k> m,$ with reproducing kernel $K_{k}(z,w)=\frac{1}{(1-\langle z,w\rangle)^{k}}$, where $z,w$ in $\mathbb{B}^{m}$.

\begin{lemma}
Let $m$-tuple $\mathbf{M}_{z}^{*}=(M_{z_{1}}^{*},\ldots,M_{z_{m}}^{*})$ be the adjoint of the $m$-tuple of multiplication operators on Hilbert space $A_{k}^{2}$ with $k> m$, then $\mathcal{A}^{\prime}(\mathbf{M}_{z}^{*})\cong H^{\infty}(\mathbb{B}^{m}).$
\end{lemma}
\begin{proof}
Letting $X^{*}\in\mathcal{A}^{\prime}(\mathbf{M}_{z}^{*})$, then for any $w=(w_{1},\ldots,w_{1})\in\mathbb{B}^{m}$ and $c\in\mathbb{C}$, we have
$$(M_{z_{i}}^{*}-\overline{w}_{i})X^{*}K(\cdot,w)c=X^{*}(M_{z_{i}}^{*}-\overline{w}_{i})K(\cdot,w)c=0,\quad 1\leq i\leq m.$$
Thus $X^{*}$ maps $K(\cdot,w)\mathbb{C}$ into itself. Furthermore, there is a bounded linear function $\phi(w)$ on complex field $\mathbb{C}$ such that $X^{*}K(\cdot,w)c=K(\cdot,w)\phi(w)^{*}c.$
At this time, $$\langle(Xh)(w),c\rangle_{\mathbb{C}}=\langle h,X^{*}K(\cdot,w)c\rangle_{A_{k}^{2}}=\langle h,K(\cdot,w)\phi(w)^{*}c\rangle_{A_{k}^{2}}=\langle \phi(w)h(w),c\rangle_{\mathbb{C}}$$ for any $h\in A_{k}^{2}$.
Thus $X^{*}=M_{\phi}^{*}$. From $X^{*}\ker(M_{z_{i}}^{*}-\overline{w}_{i})\subset\ker(M_{z_{i}}^{*}-\overline{w}_{i})$, we know that $\phi$ is holomorphic.
Therefore, we can define the mapping
$$\Upsilon:\mathcal{A}^{\prime}(\mathbf{M}_{z}^{*})\longrightarrow H^{\infty}(\mathbb{B}^{m})$$ given by $\Upsilon(M_{\phi}^{*})=\phi.$
According to the above, the mapping $\Upsilon$ is a injective homomorphism.
The following shows that $\Upsilon$ is a surjection.
For any $\psi\in H^{\infty}(\mathbb{B}^{m})$, let its power series expansion be $\psi(z)=\sum\limits_{\alpha\in\mathbb{N}^{m}}\widehat{\psi}(\alpha)z^{\alpha}$, and $\{\mathbf{e}_{\alpha}(z)\}_{\alpha\in\mathbb{N}^{m}}$ be an orthonormal basis of space $A_{k}^{2}$.
From Theorem 41 in \cite{ZZ2008}, we can know that $\mathbf{e}_{\alpha}(z)=\sqrt{\frac{\Gamma(k+|\alpha|)}{\alpha!\Gamma(k)}}z^{\alpha}$, where $\Gamma$ is the gamma function. Hence
$$\Vert\psi(z)\Vert^{2}_{A_{k}^{2}}=
\sum\limits_{\alpha\in\mathbb{N}^{m}}|\widehat{\psi}(\alpha)|^{2}\langle z^{\alpha},z^{\alpha}\rangle_{A_{k}^{2}}=\sum\limits_{\alpha\in\mathbb{N}^{m}}|\widehat{\psi}(\alpha)|^{2}\frac{\alpha!\Gamma(k)}
{\Gamma(k+|\alpha|)}\leq\sum\limits_{\alpha\in\mathbb{N}^{m}}|\widehat{\psi}(\alpha)|^{2}<\infty.$$
This means that $\Vert M_{\psi}^{*}\Vert=\Vert\psi(z)\Vert_{A_{k}^{2}}<\infty$ and $M_{\psi}^{*}\in \mathcal{A}^{\prime}(\mathbf{M}_{z}^{*}).$
\end{proof}

\begin{corollary}\label{0.11}
Let $\mathbf{T}=\mathbf{M}_{z}^{*(n)}$, then $\Gamma_{\mathbf{T}}$ is an isometry isomorphism from $\mathcal{A}^{\prime}(\mathbf{T})$ onto $M_{n}(H^{\infty}(\mathbb{B}^{m})).$
\end{corollary}

\begin{thm}\label{0.12}
Let $\mathbf{A}=(A_{1},\ldots,A_{m})\in\mathbf{\mathcal{B}}_{1}(\mathbb{B}^{m})$ be a commutative $m$-tuple, $\mathbf{T}=\mathbf{A}^{(n)}$ and $P\in\mathcal{A}^{\prime}(\mathbf{T})$ be an idempotent, denote $\mathbf{T}_{1}=\mathbf{T}|_{P\mathcal{H}^{(n)}}$. If $\mathbf{T}_{1}\in\mathbf{\mathcal{B}}_{l}(\mathbb{B}^{m})$, then $\mathbf{T}_{1}\sim_{u}\mathbf{A}^{(l)}.$
\end{thm}
\begin{proof}
Since $m$-tuples $\mathbf{A}$ and $\mathbf{M}_{z}^{*}$ are in $\mathbf{\mathcal{B}}_{1}(\mathbb{B}^{m})$, we can find the holomorphic frames $v(z)$ and $e(z)$ of $\ker(\mathbf{A}-zI)$ and $\ker(\mathbf{M}^{*}_{z}-zI)$ respectively. That is
$$(\mathbf{A}-zI)v(z)=0,\quad (\mathbf{M}^{*}_{z}-zI)e(z)=0$$
for $z=(z_{1},\ldots,z_{m})\in\mathbb{B}^{m}$. Set
$$v_{k}(z):=(\overbrace{ 0,\ldots,0}^{k-1},v(z),0,\ldots,0)^{T}$$
and
$$e_{k}(z)=(\overbrace{ 0,\ldots,0}^{k-1},e(z),0,\ldots,0)^{T},$$
where $k=1,2,\ldots,n$, $z\in\mathbb{B}^{m}$ and $T$ stand for matrix transpose.
Let
$$P(z)=(\Gamma_{\mathbf{T}}P)(z),\quad z\in\mathbb{B}^{m},$$
then $P(z)=(P_{ij}(z))_{n\times n}\in M_{n}(H^{\infty})$ is an idempotent. By Lemma \ref{0.11}, $P(z)\in\mathcal{A}^{\prime}(\mathbf{M}_{z}^{*(n)})$ is an idempotent. Set
$$Q:=P(z) \quad \text{and} \quad \mathbf{S}:=\mathbf{M}_{z}^{*(n)}|_{Q\mathcal{H}^{(n)}}.$$
From $\mathbf{T}_{1}\in\mathbf{\mathcal{B}}_{l}(\mathbb{B}^{m})$, $Q\ker \mathbf{M}_{z}^{*(n)}=\ker \mathbf{S}$ and $P\ker \mathbf{T}=\ker \mathbf{T}_{1},$
$$\dim\ker \mathbf{S}=\dim ran Q=\dim ran P(0)=\dim\ker \mathbf{T}_{1}=l.$$
By Lemma \ref{0.8}, there exists a unitary operator $U$ such that
$$U(Q\mathcal{H}^{(n)})=\mathcal{H}^{(l)}\oplus0^{(n-l)},$$
that is,
\begin{equation}\label{1}
UQU^{*}=\begin{pmatrix} I_{\mathcal{H}^{(l)}}\,\, & \star\\ 0 \,\,& 0\end{pmatrix}\begin{matrix} \mathcal{H}^{(l)}\\ \quad \mathcal{H}^{(n-l)} \end{matrix}.
\end{equation}
Let $V=U|_{Q\mathcal{H}^{(n)}}$, then
$V\mathbf{S}=\mathbf{M}_{z}^{*(l)}V.$
It follows that
$$V^{*}(\mathcal{H}^{(l)}\oplus0^{(n-l)})=U^{*}(\mathcal{H}^{(l)}\oplus0^{(n-l)})=Q\mathcal{H}^{(n)}$$
and
$$U^{*}e_{i}(z)\in\ker(\mathbf{S}-z)\subset\ker(\mathbf{M}_{z}^{*(n)}-z)=\bigcap\limits_{j=1}^{m}\ker (M_{z_{j}}^{*(n)}-z_{j}),\quad 1\leq i\leq l.$$
Because $\{e_{1}(z),e_{1}(z),\ldots,e_{n}(z)\}$ is a holomorphic frame of $\ker(\mathbf{M}_{z}^{*(n)}-z)$, we get
$$U^{*}e_{i}(z)=\lambda_{i1}(z)e_{1}(z)+\cdots+\lambda_{in}(z)e_{n}(z),\quad 1\leq i\leq l,$$
where $\lambda_{ij}(z)\in\mathbb{C}$. Note that
$$\langle e_{i}(z), e_{j}(z)\rangle=\delta_{ij}\langle e(z), e(z)\rangle,\quad1\leq i,j\leq n$$
and $U^{*}$ is a unitary operator.
Thus
\begin{equation}\label{2}
\lambda_{i1}(z)\overline{\lambda_{j1}(z)}+\cdots+\lambda_{in}(z)\overline{\lambda_{jn}(z)}=\delta_{ij},\quad 1\leq i,j\leq l, \,\,z\in \mathbb{B}^{m}.
\end{equation}
From (\ref{1}), we know that
$$UQU^{*} e_{i}(z)= UP(z)U^{*}e_{i}(z)=I_{\mathcal{H}^{(l)}}e_{i}(z)=e_{i}(z),\quad  1\leq i\leq l, \,\,z\in \mathbb{B}^{m}.$$
Therefore, $P(z)U^{*}e_{i}(z)=U^{*}e_{i}(z)$, that is,
\begin{equation}\label{3}
(P_{ij}(z))_{n\times n}(\lambda_{i1}(z),\lambda_{i1}(z),\ldots,\lambda_{i1}(z))^{\mathcal{T}}=(\lambda_{i1}(z),\lambda_{i1}(z),\ldots,\lambda_{i1}(z))^{\mathcal{T}},
\end{equation}
where $1\leq i\leq l$ and $ z\in \mathbb{B}^{m}.$ Letting
$$w_{i}(z):=\lambda_{i1}(z)v_{1}(z)+\cdots+\lambda_{in}(z)v_{n}(z),\quad 1\leq i\leq l.$$
From (\ref{2}), (\ref{3}) and
\begin{equation}\label{4}
\langle v_{i}(z), v_{j}(z)\rangle=\delta_{ij}\langle v(z), v(z)\rangle,\quad1\leq i,j\leq n,
\end{equation}
We know that for $1\leq i,j\leq l,$ there are
$\langle w_{i}(z), w_{j}(z)\rangle=\delta_{ij}\langle v(z), v(z)\rangle$
and
$P(z)w_{i}(z)=w_{i}(z).$
Since $P(z)\ker(\mathbf{T}-z)=\ker(\mathbf{T}_{1}-z)$ and $\mathbf{T}_{1}\in\mathbf{\mathcal{B}}_{l}(\mathbb{B}^{m})$,
we have $$w_{i}(z)\in\ker(\mathbf{T}_{1}-z),\quad 1\leq i\leq l$$ and $\{w_{1}(z),w_{2}(z),\ldots,w_{l}(z)\}$ forms a holomorphic frame of $\ker(\mathbf{T}_{1}-z)$ for each $z\in \mathbb{B}^{m}.$ For $z\in \mathbb{B}^{m},$ define
$$W(z): \ker(\mathbf{A}^{(l)}-z)\longrightarrow\ker(\mathbf{T}_{1}-z)$$
as follows
\begin{equation}\label{5}
W(z)v_{i}(z)=w_{i}(z),\quad 1\leq i\leq l.
\end{equation}
By (\ref{4}) and (\ref{5}),
$$\langle W(z) v_{i}(z),W(z) v_{j}(z)\rangle=\langle w_{i}(z), w_{j}(z)\rangle=\delta_{ij}\langle v(z), v(z)\rangle=\langle v_{i}(z), v_{j}(z)\rangle,\quad1\leq i,j\leq l.$$
Thus $W(z)$ is a holomorphic isometric bundle map, using the Rigidity Theorem, we have
$$\mathbf{T}_{1}\sim_{u}\mathbf{A}^{(l)}.$$
\end{proof}

\begin{thm}\label{024}
Let $\mathbf{T}=(T_{1},\ldots,T_{m})\in\mathbf{\mathcal{B}}_{1}(\mathbb{B}^{m})\cap\mathcal{L}(\mathcal{H})^{m}$. Then $\bigvee(\mathcal{A}^{\prime}(\mathbf{T}))\cong\mathbb{N}$ and $K_{0}(\mathcal{A}^{\prime}(\mathbf{T}))\cong\mathbb{Z}$.
\end{thm}
\begin{proof}
From Lemma \ref{6}, we know that $\mathbf{T}=(T_{1},\ldots,T_{m})$ is strongly irreducible. For every natural number $n$ and idempotent $P\in\bigvee(\mathcal{A}^{\prime}(\mathbf{T}^{(n)}))$,
if $\mathbf{A}_{1}=\mathbf{T}|_{P\mathcal{H}^{(n)}}=(T_{1}|_{P\mathcal{H}^{(n)}},\ldots,T_{m}|_{P\mathcal{H}^{(n)}})\in\mathbf{\mathcal{B}}_{1}(\mathbb{B}^{m})$,
then $\mathbf{A}_{1}\sim_{u}\mathbf{T}$ is obtained from Theorem \ref{0.12}, and $\mathbf{A}_{1}$ is also strongly irreducible. Therefore, from Theorem \ref{2-1}, $\bigvee(\mathcal{A}^{\prime}(\mathbf{T}))\cong\mathbb{N}$. Note that $K_{0}(\mathbb{N})=\mathbb{Z}$, then $K_{0}(\mathcal{A}^{\prime}(\mathbf{T}))\cong\mathbb{Z}$.
\end{proof}

\begin{corollary}\label{025}
Let $\mathbf{T}=(T_{1},\ldots,T_{m}),\widetilde{\mathbf{T}}=(\widetilde{T}_{1},\ldots,\widetilde{T}_{m})\in\mathbf{\mathcal{B}}_{1}(\mathbb{B}^{m})\cap\mathcal{L}(\mathcal{H})^{m}$. Then
$\mathbf{T}\sim_{s}\widetilde{\mathbf{T}}$ if and only if $K_{0}(\mathcal{A}^{\prime}(\mathbf{T}\oplus\widetilde{\mathbf{T}}))\cong\mathbb{Z}$.
\end{corollary}
\begin{proof}
This conclusion can be directly drawn from Theorem \ref{2-1}.
\end{proof}

\begin{corollary}
Let $\mathbf{T}\in\mathbf{\mathcal{B}}_{1}(\mathbb{B}^{m})\cap\mathcal{L}(\mathcal{H})^{m}$. Then
$\bigvee\left(\text{Mult}(\mathcal{H})\right)\cong \mathbb{N}$ and $K_{0}(\text{Mult}(\mathcal{H}))=\mathbb{Z}$.
\end{corollary}

\begin{corollary}
$\bigvee\left(H^{\infty}(\mathbb{B}_{m})\right)\cong \mathbb{N},~ K_{0}(H^{\infty}(\mathbb{B}_{m}))=\mathbb{Z}$.
\end{corollary}

\end{document}